\documentclass{amsart}
\usepackage[english]{babel}

\usepackage[a4paper,top=2cm,bottom=2cm,left=3cm,right=3cm,marginparwidth=1.75cm]{geometry}

\usepackage{amsmath}
\usepackage{graphicx}
\usepackage{amssymb}
\usepackage[colorlinks=true, allcolors=blue]{hyperref}
\usepackage{hyperref}

 \newtheorem{thm}{Theorem}[section]

 \newtheorem{prop}[thm]{Proposition}
\theoremstyle{definition}

 \theoremstyle{remark}

\numberwithin{equation}{section}

\title[Variation and oscillation operators in Schr\"odinger setting]{Variation and oscillation operators on weighted Morrey-Campanato spaces in the Schr\"odinger setting}
\author{V. Almeida, J.J. Betancor, J.C. Fari\~na and L. Rodr\'{\i}guez-Mesa}
\address{V\'{\i}ctor Almeida, Jorge J. Betancor, Juan C. Fari\~na and Lourdes Rodr\'{\i}guez-Mesa\newline
	Departamento de An\'alisis Matem\'atico, Universidad de La Laguna,\newline
	Campus de Anchieta, Avda. Astrof\'isico S\'anchez, s/n,\newline
	38721 La Laguna (Sta. Cruz de Tenerife), Spain}
\email{valmeida@ull.edu.es, jbetanco@ull.es, jcfarina@ull.edu.es,
lrguez@ull.edu.es
}

\thanks{The authors are partially supported by grant PID2019-106093GB-I00 from the Spanish Government}


\dedicatory{Dedicated to the memory of our friend and colleague Eleanor Harboure}
\date{\today}
\begin{document}
\maketitle

\begin{abstract}
We denote by $\mathcal{L}$ the Schr\"odinger operator with potential $V$, that is, $\mathcal{L}=-\Delta+V$, where it is assumed that $V$ satisfies a reverse H\"older inequality. We consider weighted Morrey-Campanato spaces ${\rm BMO}_{\mathcal{L},w}^\alpha (\mathbb{R}^d)$ and ${\rm BLO}_{\mathcal{L},w}^\alpha (\mathbb{R}^d)$ in the Schr\"odinger setting. We prove that the variation operator $V_\sigma (\{T_t\}_{t>0})$, $\sigma >2$, and the oscillation operator $O(\{T_t\}_{t>0},\{t_j\}_{j\in\mathbb{Z}})$, where $t_j<t_{j+1}$, $j\in \mathbb{Z}$, $\displaystyle \lim_{j\rightarrow +\infty}t_j=+\infty$ and $\displaystyle\lim_{j\rightarrow -\infty }t_j=0$, being $T_t=t^k\partial _t^ke^{-t\mathcal{L}}$, $t>0$, with $k\in \mathbb{N}$, are bounded operators from ${\rm BMO}_{\mathcal{L},w}^\alpha (\mathbb{R}^d)$ into ${\rm BLO}_{\mathcal{L},w}^\alpha (\mathbb{R}^d)$. We also establish the same property for the maximal operators defined by $\{t^k\partial _t^ke^{-t\mathcal L}\}_{t>0}$, $k\in \mathbb{N}$.
\end{abstract}

\section{Introduction}\label{S1}

Let $\{T_t\}_{t>0}$ be a family of bounded operators in $L^p(\mathbb R^d)$ for some $1 \leq p < \infty$. Many times we are interested in knowing the behavior of $T_t$ when $ t\rightarrow 0^+$. Specifically we want to know if there exists the limit $\displaystyle\lim_{t \rightarrow 0^+}T_t(f)(x)$ for almost everywhere $x \in \mathbb R^d$ when $f\in L^p(\mathbb R^d)$. A first way to deal with the problem is to consider the maximal operator $T_*$ defined by $\displaystyle T_*f=\sup_{t>0} |T_tf|$. If $T_*$ defines a bounded operator from $L^p(\mathbb R^d)$ into $L^{p,\infty}(\mathbb R^d)$ and $\displaystyle\lim_{t\rightarrow 0^+}T_t(g)(x)$ exists for almost all $x\in \mathbb R^d$ when $g \in \mathcal D$ where $\mathcal D$ is a dense subspace of $L^p(\mathbb R^d)$, then $\displaystyle\lim_{t \rightarrow 0^+}T_t(t)(x)$ exists for almost all $x\in \mathbb R^d$ when $f \in L^p(\mathbb R^d)$. This procedure is well known and it is named Banach principle (\cite[pp.27-28]{Duo}). Other approach to study this question is based in the variation operator. Let $\sigma >2$. The variation operator  $V_\sigma(\{T_t\}_{t>0})$ is defined by

$$
V_\sigma(\{T_t\}_{t>0})(f)(x)= \sup_{\substack{0<t_n<t_{n-1}<\dots<t_1\\n\in \mathbb N}}
\left( \sum^{n-1}_{j=1}|T_{t_{j+1}}(f)(x) - T_{t_j}(f)(x)|^\sigma\right)^{\frac{1}{\sigma}}.
$$
If $V_\sigma(\{T_t\}_{t>0})(f)(x) < \infty$, then there exists the limit $\displaystyle\lim_{t \rightarrow 0^+}T_t(f)(x)$.

We observe that in this case it is not necessary to have the existence of the limit when $f$ is in a dense subset of $L^p(\mathbb R^d)$. In order to see the measurability of $V_\sigma(\{T_t\}_{t>0})(f)$ when $f \in L^p(\mathbb R^d)$ we need additional properties for $\{T_t\}_{t>0}$. For instance, if for almost all $x\in \mathbb R^d$ the function $t \longrightarrow T_t(f)(x)$ is continuous in $(0,\infty)$.
Then
$$
V_\sigma(\{T_t\}_{t>0})(f)(x) = \sup_{\substack{0<t_n<t_{n-1}<\ldots<t_1\\ t_j\in \mathbb Q,j=1,\,\ldots,n\\
n\in \mathbb N}}\left( \sum^{n-1}_{j=1}|T_{t_{j+1}}(f)(x) - T_{t_j}(f)(x)|^\sigma\right)^{\frac{1}{\sigma}},\;\;\mbox { a.e. }x\in \mathbb R^d,
$$
and $V_\sigma(\{T_t\}_{t>0}(f)$ is measurable in $\mathbb R^d$. Once the measurability property is assumed it is of interested to study the boundedness of the variation operators in function spaces. Note that if $V_\sigma(\{T_t\}_{t>0})(f)(x)$ defines a bounded operator in $L^p$, BMO, Lipschitz or Hardy spaces, for instance, then $V_\sigma(\{T_t\}_{t>0})(f)(x) < \infty$ for almost all $x\in\mathbb R^d$, when $f$ belongs to those function spaces. Furthermore, the boundedness properties of the variation operator inform us about the speed of convergence of $T_t(f)(x)$ as $t \rightarrow 0^+$.

Variational inequalities have been very studied in the last two decads in probability, ergodic theory and harmonic analysis. Lépingle (\cite{Le}) established the first variational inequality involving martingales improving the classical Doob maximal inequality. Bourgain (\cite{Bo}), some years later, proved a variational inequality for the ergodic average of a dynamic system. Since then many authors have studied variation operators in harmonic analysis (see, for instance, \cite{BORSS}, \cite{CJRW1}, \cite{CJRW2}, \cite{DMT}, \cite{JSW}, \cite{MTXu}, \cite{MT}, \cite{MSZ}, \cite{MTZ} and  \cite{OSTTW}).

In order to  obtain $L^p$-variation inequalities it is usual to need $\sigma >2$ (see \cite[Remark 1.7]{CJRW1} and \cite{Q}). When $\sigma =2$ a good substitute is the oscillation operator defined as follows. Suppose that $\{t_j\}_{j\in \mathbb Z}$ is a sequence of positive numbers such that $0<t_j<t_{j+1}<\infty$, $j \in \mathbb Z$, $\displaystyle \lim_{j \rightarrow -\infty}t_j=0$ and $\displaystyle \lim_{j \rightarrow +\infty}t_j=+\infty$. We define the oscillation operator associated with $\{t_j\}_{j\in \mathbb Z}$ for $\{T_t\}_{t>0} $ by
$$
O(\{T_t\}_{t>0}, \{t_j\}_{j \in \mathbb Z})(f)(x) = \left( \sum_{j\in \mathbb Z}\sup_{t_j \leq \varepsilon_j < \varepsilon_{j+1}< t_{j+1}}|T_{\varepsilon_j}(f)(x) - T_{\varepsilon_{j+1}}(f)(x)|^2\right)^{\frac{1}{2}}.
$$
Note that if the exponent 2 in the last definition is replaced by other greater than 2 the new operator is controlled by that with exponent 2.

Finally we recall the definition of the short variation operator $SV(\{T_t\}_{t>0})$. For every $k \in \mathbb Z$ we define
$$
V_k(\{T_t\}_{t>0})(f)(x) = \sup_{\substack{2^{-k}<t_n<\ldots < t_1\leq 2^{-k+1}\\n\in \mathbb N}}\left( \sum^{n-1}_{j=1}|T_{t_j}(f)(x)-T_{t_{j+1}}(f)(x)|^2 \right)^{\frac{1}{2}}.
$$
The short variation operator $SV(\{T_t\}_{t>0})$ is given by
$$
SV(\{T_t\}_{t>0})(f)(x)=\left( \sum_{k\in \mathbb Z}(V_k(\{T_t\}_{t>0})(f)(x))^2\right)^{\frac{1}{2}}.
$$
Our objective in this paper is to study the variation, oscillation and short variation operators when $T_t=t^k\partial^k_tS_t$, $t>0$, with $k\in \mathbb N$, where $\{S_t\}_{t>0}$ represents the heat or Poisson semigroup associated with the Schr\"odinger operator in $\mathbb R^d$. We consider weighted Morrey-Campanato spaces in the Schr\"odinger setting.

We denote by $\mathcal L$ the Schr\"odinger operator in $\mathbb R^d$, $d \geq 3$, defined by
$$
\mathcal L= - \Delta + V,
$$
where $\displaystyle\Delta = \sum_{i=1}^d \partial^2_{x_i}$ represents the Euclidean Laplacian and the potential $V \geq 0$ is not identically zero and it belongs to $q$-reverse H\"older class (in short, $V\in RH_q(\mathbb R^d)$), that is, there exists $C>0$ such that
$$
\left(\frac{1}{|B|}\int_BV(x)^qdx\right)^{\frac{1}{q}} \leq \frac{C}{B}\int_BV(x)dx,
$$
for every ball $B$ in $\mathbb R^d$. The class $RH_q(\mathbb R^d)$, is defined in this way for $1<q<\infty$. Every nonegative polynomial is in $RH_q(\mathbb R^d)$ for each $1<q<\infty$.

Harmonic analysis associated with the operator $\mathcal L$ has been developed by several authors in the century. Shen's paper \cite{Sh1} can be considered the starting point of the most of these studies (see, for instance, \cite{DYZ}, \cite{DGMTZ}, \cite{DZ}, \cite{Ky}, \cite{MSTZ}, \cite{Sh2} and \cite{WY}). Professor Eleanor Harboure, to whose memory this paper is dedicated, studied several important aspects of the harmonic analysis in the Schr\"odinger setting (\cite{BFHR1}, \cite{BFHR2}, \cite{BCH1}, \cite{BCH2}, \cite{BCH3}, \cite{BHQ1}, \cite{BHQ4}, \cite{BHQ3}, \cite{BHQ2}, \cite{BHQ5}, \cite{BHS1}, \cite{BHS5}, \cite{BHS2}, \cite{BHS3}, \cite{BHS4}, \cite{BBHV} and \cite{HSV}).

The following function $\rho$, that is named critical radius, plays an important role and it is defined by
$$
\rho (x)=\sup\big\{r\in (0,\infty ): \frac{1}{r^{d-2}}\int_{B(x,r)}V(y)dy\leq 1\big\}.
$$

The Schr\"odinger operator $\mathcal L$ becomes a nice perturbation of the Euclidean Laplacian, that means that the harmonic analysis operators (Riesz transforms, multipliers, Littlewood-Paley functions) have the same behaviour close to the diagonal than the corresponding Euclidean operators. The closeness to the diagonal is defined by the critical radius function, The main properties of the function $\rho$ were established in \cite[Lemma 1.4]{Sh1}.

By a weight $w$ we understand a measurable and positive function in $\mathbb R^d$. As in \cite{BHS2} we say that a weight $w$ is in $A^{\rho,\theta}_p(\mathbb R^d)$, with $1<p<\infty$ and $\theta >0$, when there exists $C>0$ such that, for every ball $B$ in $\mathbb R^d$,
$$
\left(\frac{1}{\Psi_\theta(B)|B|}\int_B w(y) dy\right)\left(\frac{1}{\Psi_\theta(B)|B|}\int_B w^{-\frac{1}{p-1}}(y)dy\right)^{p-1} \leq C.
$$
Here if $x \in \mathbb R^d$ and $r>0$
$$
\Psi_\theta(B(x,r))= \left(1+\frac{r}{\rho(x)}\right)^\theta.
$$
We define $A_p^{\rho,\infty}(\mathbb R^d)= \displaystyle\cup_{\theta >0}A_p^{\rho,\theta}(\mathbb R^d)$, $1<p<\infty$.

In \cite{BHS2} and \cite{Ta} the main properties of the weights in $A_p^{\rho ,\infty }(\mathbb R^d)$ were proved.

We now define the Morrey-Campanato spaces ${\rm BMO}_{\mathcal L,w}^\alpha (\mathbb R^d)$ and ${\rm BLO}_{\mathcal L,w}^\alpha (\mathbb R^d)$.

Let $w\in A_p^{\rho ,\infty }(\mathbb R^d)$ and $\alpha \in [0,1)$. A locally integrable function $f$ on $\mathbb R^d$ is said to be in ${\rm BMO}_{\mathcal L,w}^\alpha (\mathbb R^d)$ when there exists $C>0$ such that
\begin{equation}\label{1.1}
\frac{1}{|B(x_0,r_0)|^\alpha w(B(x_0,r_0))}\int_{B(x_0,r_0)}|f(y)-f_{B(x_0,r_0)}|dy\leq C,\quad x_0 \in \mathbb R^d\mbox{ and }0<r_0<\rho (x_0),
\end{equation}
where 
$$
f_{B(x_0,r_0)}=\frac{1}{|B(x_0,r_0)|}\int_{B(x_0,r_0)}f(y)dy,\quad x_0\in \mathbb R^d\mbox{ and }r_0>0,
$$
and
\begin{equation}\label{1.2}
\frac{1}{|B(x_0,r_0)|^\alpha w(B(x_0,r_0))}\int_{B(x_0,r_0)}|f(y)|dy\leq C,\quad x_0 \in \mathbb R^d\mbox{ and }r_0\geq\rho (x_0),
\end{equation}

We define 
$$
\|f\|_{{\rm BMO}_{\mathcal L,w}^\alpha (\mathbb R^d)}=\inf\big\{C>0: \eqref{1.1}\mbox{ and }\eqref{1.2}\mbox{ hold}\big\}.
$$
As it is proved in \cite[Lemma 2.1]{TZ} in \eqref{1.2} it is sufficient to consider $r_0=\rho (x_0)$.

We say that a function $f\in {\rm BMO}_{\mathcal L,w}^\alpha (\mathbb R^d)$ is in ${\rm BLO}_{\mathcal L,w}^\alpha (\mathbb R^d)$ when there exists $C>0$ such that
\begin{equation}\label{1.3}
\frac{1}{|B(x_0,r_0)|^\alpha w(B(x_0,r_0))}\int_{B(x_0,r_0)}(f(y)-{\rm ess}\hspace{-2mm}\inf_{\hspace{-5mm}z\in B(x_0,r_0)}f(z))dy\leq C,\quad x_0 \in \mathbb R^d\mbox{ and }0<r_0<\rho (x_0).
\end{equation}
We define 
$$
\|f\|_{{\rm BLO}_{\mathcal L,w}^\alpha (\mathbb R^d)}=\inf\big\{C>0: \eqref{1.2} \mbox{ and }\eqref{1.3}\mbox{ hold} \big\}.
$$ 
It is clear that ${\rm BLO}_{\mathcal L,w}^\alpha (\mathbb R^d)$ is contained in ${\rm BMO}_{\mathcal L,w}^\alpha (\mathbb R^d)$.

Note that the spaces ${\rm BMO}_{\mathcal L,w}^\alpha (\mathbb R^d)$ and ${\rm BLO}_{\mathcal L,w}^\alpha (\mathbb R^d)$ actually depend on the critical radius function $\rho$ but here we prefer to point out the dependence of the operator $\mathcal L$.

The operator $-\mathcal L$ generates a semigroup of operators $\{W_t^{\mathcal L}:=e^{-t\mathcal L}\}_{t>0}$ on $L^p(\mathbb R^d)$, $1\leq p<\infty$, where, for every $t>0$,
$$
W_t^{\mathcal L}(f)(x)=\int_{\mathbb R^d}W_t^\mathcal L(x,y)f(y)dy,\quad f\in L^p(\mathbb R^d),\;1\leq p<\infty.
$$
$\{W_t^\mathcal L\}_{t>0}$ is also named the heat semigroup associated with $\mathcal L$. For every $t>0$, $W_t^{\mathcal L} (\cdot, \cdot )$ is a positive symmetric function on $\mathbb R^d\times \mathbb R^d$ and satisfies that $\int_{\mathbb R^d}W_t^\mathcal L (x,y)dy\leq 1$. The semigroup $\{W_t^\mathcal L\}_{t>0}$ is not Markovian.

By using subordination formula (\cite[pp. 259-268]{Yo}), for every $\beta \in (0,1)$, the semigroup of operators $\{W_{\beta, t}^\mathcal L\}_{t>0}$ generated by $-\mathcal L^\beta $ is defined by
$$
W_{\beta, t}^\mathcal L(f)=\int_0^\infty \eta_t^\beta (s)W_s^\mathcal L(f)ds,\quad t>0,
$$
where $\eta_t^\beta$ is a certain nonnegative continuous function. The special case $\{W_{1/2,t}^ \mathcal L\}_{t>0}$ is known as Poisson semigroup associated with $\mathcal L$.

In \cite[Theorem 6]{DGMTZ} it was proved that the maximal operators $W_*^\mathcal L$ and $W_{1/2,*}^\mathcal L$ defined by
$$
W_*^\mathcal L(f)=\sup_{t>0}|W_t^\mathcal L (f)|\quad \mbox{ and }\quad W_{1/2,*}^\mathcal L (f)=\sup_{t>0}|W_{1/2,t}^\mathcal L(f)|,
$$
are bounded from ${\rm BMO}_{\mathcal L}(\mathbb R^d)$ into itself, where by ${\rm BMO}_{\mathcal L}(\mathbb R^d)$ we represent the space ${\rm BMO}_{\mathcal L,w}^\alpha (\mathbb R^d)$ when $w=1$ and $\alpha =0$. 
Ma, Stinga, Torrea and Zhang (\cite[Theorem 1.3]{MSTZ}) proved that $W_*^\mathcal L$ and $W_{1/2,*}^\mathcal L$ are bounded from ${\rm BMO}_{\mathcal L}^\alpha (\mathbb R^d)$ into itself, where ${\rm BMO}_{\mathcal L}^\alpha (\mathbb R^d)$ denotes the space ${\rm BMO}_{\mathcal L,w}^\alpha (\mathbb R^d)$ with $w=1$. 
In \cite[Proposition 5.2, (i)]{YYZ2} it was established that $W_*^\mathcal L$ and $W_{1/2,*}^\mathcal L$ are bounded from $E_\rho ^{\alpha ,p}(\mathbb R^d)$ into $\widetilde{E}_\rho ^{\alpha ,p}(\mathbb R^d)$, when $1<p<\infty$, and where these spaces are defined like ${\rm BMO}_{\mathcal L,w}^\alpha (\mathbb R^d)$ and ${\rm BLO}_{\mathcal L,w}^\alpha (\mathbb R^d)$, but where the $L^1$-norm is replaced by the $L^p$-norm and $w=1$.

We now consider, for every $k\in \mathbb N$, the maximal operators 
$$
W_*^{\mathcal L,k}(f)=\sup_{t>0}|t^k \partial _t^kW_t^\mathcal L (f)|.
$$

Our first result is the following.
\begin{thm}\label{Th1.1}
    Let $k\in \mathbb N$, $q>d/2$ and $\alpha \in [0,1)$. Suppose that $V\in RH_q(\mathbb{R}^d)$ and that $w\in A_p^{\rho,\theta}(\mathbb{R}^d)$, for some $\theta>0$ such that $2(d(p+\alpha-1)+p\theta)<min\{1,2-d/q\}$. Then, the maximal operators $W_*^{\mathcal L,k}$ are bounded from ${\rm BMO}_{\mathcal L,w}^\alpha (\mathbb R^d)$ into ${\rm BLO}_{\mathcal L,w}^\alpha (\mathbb R^d)$.
 \end{thm}
The variation operator $V_\sigma (\{W_t^{\mathcal L}\}_{t>0})$ was studied in \cite{BFHR1} and \cite{BFHR2}. In \cite[Theorem 2.6]{BFHR2} it was proved that $V_\sigma (\{W_t^\mathcal L\}_{t>0}$ is bounded from ${\rm BMO}_\mathcal L (\mathbb R^d)$ into itself. This result was extended by Bui (\cite{Bui}) when the Schr\"odinger operator $\mathcal L$ is replaced by other operator $L$ such that the kernel of $e^{-tL}$, $t>0$, satisfies the same properties than the kernel of $e^{-t\mathcal L}$ (see \cite[p. 125]{Bui}). Tang and Zhang (\cite{TZ}) generalized \cite[Theorem 2.6]{BFHR2} proving that $V_\sigma (\{W_t^\mathcal{L}\}_{t>0})$ is bounded from ${\rm BMO}_{\mathcal L,w}^\alpha (\mathbb R^d)$ into itself (see \cite[Theorem 5]{TZ}).
We extend this last property as follows. The theorem is a complement of the results given in \cite{ZT}. 
\begin{thm}\label{Th1.2}
    Let $k\in \mathbb N$, $q>d/2$, $\alpha \in [0,1)$, $\sigma >2$, and $1<p<\infty$. Suppose that $V\in RH_q(\mathbb{R}^d)$ and that $w\in A_p^{\rho,\theta}(\mathbb{R}^d)$, for some $\theta>0$, and $\{t_j\}_{j\in \mathbb Z}$ is a sequence of positive numbers satisfying that $t_j<t_{j+1}$, $j\in \mathbb Z$, $\lim_{j\rightarrow +\infty}t_j=+\infty$, $\lim_{j\rightarrow -\infty}t_j=0$. If $2(d(p+\alpha-1)+p\theta)<min\{1,2-d/q\}$, then the operators $V_\sigma (\{t^k\partial _t^kW_t^\mathcal L\}_{t>0})$, $O(\{t^k\partial _t^kW_t^\mathcal L\}_{t>0},\{t_j\}_{j\in \mathbb Z})$ and $SV(\{t^k\partial _t^kW_t^\mathcal L\}_{t>0})$ are bounded from ${\rm BMO}_{\mathcal L,w}^\alpha (\mathbb R^d)$ into ${\rm BLO}_{\mathcal L,w}^\alpha (\mathbb R^d)$.
\end{thm}
In the proof of Theorems \ref{Th1.1} and \ref{Th1.2} we are inspired by the  ideas developed by Da. Yang, Do. Yang and Zhou (\cite{YYZ1}, \cite{YYZ3} and \cite{YYZ2}) and Tang and Zhang (\cite{TZ}).

We organize the paper as follows. In Section \ref{S2} we recall some properties about the kernels, the weights and the spaces that will be useful in the proofs of our results. The proof of Theorem \ref{Th1.2} for the variation operator is given in Section \ref{S3}. We prove Theorem \ref{Th1.2} for the oscillation operator in Section \ref{S4}. In Section \ref{S5} we give a proof of Theorem \ref{Th1.2} for the short variation operator. A sketch of the proof of Theorem \ref{Th1.1} is presented in Section \ref{S6}.

Our arguments allow us also to prove the same properties when the semigroup $\{W_t^\mathcal{L}\}_{t>0}$ is replaced by $\{W_{\beta, t}^{\mathcal L}\}_{t>0}$, with $\beta\in (0,1)$.
We also remark that the methods we have used can be applied to establish versions of Theorems \ref{Th1.1} and \ref{Th1.2} when the operator $\mathcal L$ is replaced by the following ones:

(a) Generalized Schr\"odinger operators defined by $\mathfrak L=-\Delta +\mu$ on $\mathbb R^d$, where $\mu$ is a nonnegative Radon measure on $\mathbb R^d$ satisfying certain scale-invariant Kato condition (\cite{Sh2} and \cite{WY}).

(b) Degenerate Schr\"odinger operators on $\mathbb R^d$ defined as follows. Let $w$ belongs to the Muckenhoupt class $A_2(\mathbb R^d)$ and let $\{a_{ij}\}_{i,j=1}^d$ be a real symmetric matrix function satisfying that
$$
\frac{1}{C}|\xi |^2\leq \sum_{i,j=1}^da_{ij}(x)\xi _i\xi_j\leq C|\xi |^2,\quad x,\xi \in \mathbb R^d.
$$
The degenerate Schr\"odinger operator is defined by
$$
\mathfrak L(f)(x)=-\frac{1}{w(x)}\sum_{i,j=1}^d\partial _i(a_{ij}(\cdot )\partial _jf)(x)+V(x).
$$
Here $V$ satisfies certain integrability conditions with respect to the measure $w(x)dx$ (\cite{HSV}). 

(c) Schr\"odinger operators on $(2n+1)$-dimensional Heisenberg group $\mathbb H_n$ defined by $\mathfrak L=-\Delta_{\mathbb H^n}+V$, where $\Delta _{\mathbb H^n}$ represents the sublaplacian in $\mathbb H^n$ (\cite{LL}).

(d) Schr\"odinger operators on connected and simply connected nilpotent Lie groups $G$ defined by $\mathfrak L=-\Delta _G+V$, where $\Delta _G$ denotes the sublaplacian in $G$ (\cite{VSC}).

Throughout this paper by $c$ and $C$ we always denote positive constants that can change in each occurrence.

\section{Some auxiliary results}\label{S2}
In this section we present some results that will be useful in the sequel. We begin with some properites of the Schr\"odinger heat kernel. 
\begin{prop}\label{Prop2.1} Let $k\in \mathbb{N}$ and $q>d/2$.

(a) For every $N\in \mathbb{N}$ there exists $C=C(N)$ such that
$$
|t^k\partial_t^kW_t^\mathcal{L}(x,y)|\leq C\frac{e^{-c\frac{|x-y|^2}{t}}}{t^{d/2}}\Big(1+\frac{\sqrt{t}}{\rho (x)}+\frac{\sqrt{t}}{\rho (y)}\Big)^{-N},\quad x,y\in \mathbb R^d\mbox{ and }t>0.
$$

(b) For every $0<\delta <\min \{1,2-d/q\}$ and $N\in \mathbb{N}$ there exists $C=C(N,\delta)$ such that, for every $x,y,h\in \mathbb R^d$, $t>0$ and $|h|\leq\sqrt{t}$,
$$
|t^k\partial_t^kW_t^\mathcal{L}(x+h,y)-t^k\partial_t^kW_t^\mathcal{L}(x,y)|\leq C\frac{e^{-c\frac{|x-y|^2}{t}}}{t^{d/2}}\Big(\frac{|h|}{\sqrt{t}}\Big)^\delta \Big(1+\frac{\sqrt{t}}{\rho (x)}+\frac{\sqrt{t}}{\rho (y)}\Big)^{-N}.
$$

(c) For every $0<\delta \leq\min \{1,2-d/q\}$ and $N\in \mathbb{N}$ there exists $C=C(N,\delta)$ such that
$$
\Big|\int_{\mathbb R^d}t^k\partial_t^kW_t^\mathcal{L}(x,y)dy\Big|\leq C\Big(\frac{\sqrt{t}}{\rho (x)}\Big)^\delta\Big(1+\frac{\sqrt{t}}{\rho (x)}\Big)^{-N},\quad x\in \mathbb R^d\mbox{ and }t>0.
$$

(d) There exists $C>0$ such that
$$
|t^k\partial_t^kW_t^\mathcal{L}(x,y)-t^k\partial_t^kW_t(x-y)|\leq C\frac{e^{-c\frac{|x-y|^2}{t}}}{t^{d/2}}\Big(\frac{\sqrt{t}}{\max\{\rho (x),\rho (y)\}}\Big)^{2-\frac{d}{q}},\quad x,y \in \mathbb R^d\mbox{ and }t>0.
$$
Here, $W_t$ represents the classical heat kernel.
\end{prop}
\begin{proof}
The properties $(a)$, $(b)$ and $(c)$ were proved in \cite[Proposition 3.3]{HLL}. The property $(d)$ was established in \cite[Proposition 1]{WLZ}.
\end{proof}

In the sequel we denote $\delta_0:=\min\{1,2-d/q\}$.

We now list the main properties of the weights in $A_p^{\rho ,\theta}(\mathbb R^d)$.
\begin{prop}(\cite[Lemma 2.2]{Ta}, \cite[Proposition 2.4]{TZ})\label{Prop2.3}
Let $1<p<\infty$ and $\theta >0$. 

(a) $w\in A_p^{\rho,\theta}(\mathbb R^d)$ if, and only if, $w^{-\frac{1}{p-1}}\in A_{p'}^{\rho ,\theta}(\mathbb R^d)$, where $p'=\frac{p}{p-1}$.

(b) If $w\in A_p^{\rho,\theta}(\mathbb R^d)$, there exists $C>0$ such that
$$
\frac{w(B)}{w(E)}\leq C\Big(\frac{\psi _\theta (B)|B|}{|E|}\Big)^p,
$$
for every ball $B$ in $\mathbb{R}^d$ and every measurable set $E\subset B$.

(c) If $w\in A_p^{\rho,\theta}(\mathbb R^d)$, for every $c\geq 1$, there exists $C>0$ such that
$$
\frac{w(2^kB)}{w(B)}\leq C 2^{kp(\theta +d)},
$$
for every $k\in \mathbb{Z}$ and every ball $B=B(x,r)$ being $r\leq c\rho (x)$.
\end{prop}

Concerning to Morrey-Campanato spaces ${\rm BMO}_{\mathcal{L},w}^\alpha (\mathbb R^d)$ we will use the following result.
\begin{prop}(\cite[Corollary 2.1]{TZ})\label{Prop2.4}
Let $1<p<\infty$, $\theta >0$, $\alpha \in [0,1)$, $\nu\in (1,p']$, and $w\in A_p^{\rho ,\theta }(\mathbb R^d)$. For every $c\geq 1$, there exist $C>0$ such that, if $f\in {\rm BMO}_{\mathcal{L},w}^\alpha (\mathbb R^d)$ then 
\begin{equation}\label{2.1}
\frac{1}{|B|^\alpha }\left(\frac{1}{w(B)}\int_B|f(y)-f_B|^\nu w(y)^{1-\nu}dy\right)^{1/\nu}\leq C\|f\|_{{\rm BMO}_{\mathcal{L},w}^\alpha (\mathbb R^d)},
\end{equation}
for every $B=B(x,r)$ being $0<r\leq c\rho (x)$, and, for a certain $\gamma>0$,
\begin{equation}\label{2.2}
\frac{1}{|B|^\alpha }\left(\frac{1}{w(B)}\int_B|f(y)|^\nu w(y)^{1-\nu}dy\right)^{1/\nu}\leq C\Big(1+\frac{r}{\rho (x)}\Big)^{\gamma}\|f\|_{{\rm BMO}_{\mathcal{L},w}^\alpha (\mathbb R^d)}
\end{equation}
for every $B=B(x,r)$ with $r\geq \rho (x)$. 

\end{prop}

\section{Proof of Theorem \ref{Th1.2} for the variation operator $V_\sigma (\{t^k\partial_t^kW_t^\mathcal{L}\}_{t>0})$}\label{S3}

We have to see that there exists $C>0$ such that, for every $f\in {\rm BMO}_{\mathcal L,w}^\alpha (\mathbb R^d)$,

(i) for every $x_0\in \mathbb R^d$,
$$
\int_B|V_\sigma (\{t^k\partial _t^kW_t^\mathcal L\}_{t>0})(f)(x)|dx\leq C|B|^\alpha w(B)\|f\|_{{\rm BMO}_{\mathcal L,w}^\alpha(\mathbb R^d)},
$$
where $B=B(x_0,\rho (x_0))$;

(ii) for each $x_0\in \mathbb R^d$ and $0<r<\rho (x_0)$,
$$
\int_B(V_\sigma (\{t^k\partial _t^kW_t^\mathcal L\}_{t>0})(f)(x)-\alpha (B,f))dx\leq C|B|^\alpha w(B)\|f\|_{{\rm BMO}_{\mathcal L,w}^\alpha(\mathbb R^d)},
$$
where $\alpha (B,f)={\rm ess}\inf_{y\in B} V_\sigma (\{t^k\partial _t^kW_t^\mathcal L\}_{t>0})(f)(y)$ and $B=B(x_0,r)$.

In \cite[Theorem 4]{TZ} it was proved the variation operator $V_\sigma(\{W_t^\mathcal{L}\}_{t>0})$ is bounded from $L^p(\mathbb{R}^d,w)$ into itself. According to Proposition \ref{Prop2.1} the $k$-th derivative $\partial_t^kW_t^\mathcal{L}(x,y)$ of the heat kernel satisfies all the properties that we need to establish, by proceeding as in the proof of \cite[Theorem 4]{TZ}, that the variation operator $V_\sigma(\{t^k\partial_t^kW_t^\mathcal{L}\}_{t>0})$ is bounded from $L^p(\mathbb{R}^d,w)$ into itself. Then, by using Proposition \ref{Prop2.1}, (a), as in the proof of \cite[Theorem 5, p. 610]{TZ}, we can see that the property (i) holds.

We are going to prove (ii). Let $f\in {\rm BMO}_{\mathcal L,w}^\alpha (\mathbb R^d)$, $x_0\in \mathbb R^d$ and $0<r_0<\rho (x_0)$. We take $0<t_n<t_{n-1}<...<t_1$. In the case that $t_{i_0+1}<8r_0^2\leq t_{i_0}$ for some $i_0\in \{1,...,n-1\}$, by understanding the sums in the suitable way when $i_0=n-1$, Minkowski inequality implies that 
\begin{align*}
\Big(\sum_{i=1}^{n-1}|t^k\partial _t^kW_t^\mathcal L(f)(x)_{|t=t_{i+1}}-t^k\partial _t^kW_t^\mathcal L(f)(x)_{|t=t_i}|^\sigma\Big)^{1/\sigma}&\\
&\hspace{-6cm}=\Big[\Big(\sum_{i=1}^{i_0-1}+\sum_{i=i_0+1}^{n-1}\Big)|t^k\partial _t^kW_t^\mathcal L(f)(x)_{|t=t_{i+1}}-t^k\partial _t^kW_t^\mathcal L(f)(x)_{|t=t_i}|^\sigma\\
&\hspace{-6cm}\quad +\big|(t^k\partial _t^kW_t^\mathcal L(f)(x)_{|t=t_{i_0+1}}-t^k\partial _t^kW_t^\mathcal L(f)(x)_{|t=8r_0^2})\\
&\hspace{-6cm}\quad +(t^k\partial _t^kW_t^\mathcal L(f)(x)_{|t=8r_0^2}-t^k\partial _t^kW_t^\mathcal L(f)(x)_{|t=t_{i_0}})\big|^\sigma \Big]^{1/\sigma}\\
&\hspace{-6cm}\leq\Big[\sum_{i=1}^{i_0-1}|t^k\partial _t^kW_t^\mathcal L(f)(x)_{|t=t_{i+1}}-t^k\partial _t^kW_t^\mathcal L(f)(x)_{|t=t_i}|^\sigma\\
&\hspace{-6cm}\quad +|t^k\partial _t^kW_t^\mathcal L(f)(x)_{|t=8r_0^2}-t^k\partial _t^kW_t^\mathcal L(f)(x)_{|t=t_{i_0}}\big|^\sigma\Big]^{1/\sigma} \\
&\hspace{-6cm}\quad +\Big[\sum_{i=i_0+1}^{n-1}|t^k\partial _t^kW_t^\mathcal L(f)(x)_{|t=t_{i+1}}-t^k\partial _t^kW_t^\mathcal L(f)(x)_{|t=t_i}|^\sigma\\
&\hspace{-6cm}\quad +\big|(t^k\partial _t^kW_t^\mathcal L(f)(x)_{|t=t_{i_0+1}}-t^k\partial _t^kW_t^\mathcal L(f)(x)_{|t=8r_0^2})\big|^\sigma \Big]^{1/\sigma},\quad x\in \mathbb R^d,
\end{align*}
and, if $8r_0^2\leq t_n$ we can write
\begin{align*}
\Big(\sum_{i=1}^{n-1}|t^k\partial _t^kW_t^\mathcal L(f)(x)_{|t=t_{i+1}}-t^k\partial _t^kW_t^\mathcal L(f)(x)_{|t=t_i}|^\sigma\Big)^{1/\sigma}\\
&\hspace{-6cm}\leq \Big(|t^k\partial _t^kW_t^\mathcal L(f)(x)_{|t=8r_0^2}-t^k\partial _t^kW_t^\mathcal L(f)(x)_{|t=t_n}|^\sigma\\
&\hspace{-6cm}\quad +\sum_{i=1}^{n-1}|t^k\partial _t^kW_t^\mathcal L(f)(x)_{|t=t_{i+1}}-t^k\partial _t^kW_t^\mathcal L(f)(x)_{|t=t_i}|^\sigma\Big)^{1/\sigma},\quad x\in \mathbb R^d.
\end{align*}
We thus deduce that
$$
V_\sigma (\{t^k\partial _t^kW_t^\mathcal L\}_{t>0})(f)\leq V_\sigma (\{t^k\partial _t^kW_t^\mathcal L\}_{t\in (0, 8r_0^2]})(f)+V_\sigma (\{t^k\partial _t^kW_t^\mathcal L\}_{t\in [8r_0^2,\infty )})(f).
$$
On the other hand, it is clear that
$$
V_\sigma (\{t^k\partial _t^kW_t^\mathcal L\}_{t>0})(f)\geq V_\sigma (\{t^k\partial _t^kW_t^\mathcal L\}_{t\in [8r_0^2,\infty )})(f).
$$
Also we have that
\begin{align*}
V_\sigma (\{t^k\partial _t^kW_t^\mathcal L\}_{t\in [8r_0^2,\infty )})(f)(x)-{\rm ess}\hspace{-2mm}\inf_{\hspace{-5mm}y\in B(x_0,r_0)}V_\sigma (\{t^k\partial _t^kW_t^\mathcal L\}_{t>0})(f)(y)&\\
&\hspace{-8cm} \leq V_\sigma (\{t^k\partial _t^kW_t^\mathcal L\}_{t\in [8r_0^2,\infty )})(f)(x)-{\rm ess}\hspace{-2mm}\inf_{\hspace{-5mm}y\in B(x_0,r_0)}V_\sigma (\{t^k\partial _t^kW_t^\mathcal L\}_{t\in [8r_0^2,\infty)})(f)(y)\\
&\hspace{-8cm} \leq {\rm ess}\hspace{-2mm}\sup_{\hspace{-5mm}z,y\in B(x_0,r_0)}\big|V_\sigma (\{t^k\partial _t^kW_t^\mathcal L\}_{t\in [8r_0^2,\infty )})(f)(z)-V_\sigma (\{t^k\partial _t^kW_t^\mathcal L\}_{t\in [8r_0^2,\infty)})(f)(y)\big|\\
&\hspace{-8cm} \leq {\rm ess}\hspace{-2mm}\sup_{\hspace{-5mm}z,y\in B(x_0,r_0)}\sup_{8r_0^2\leq t_n<...<t_1}\Big(
\sum_{i=1}^{n-1}\big|(t^k\partial _t^kW_t^\mathcal L(f)(z)_{|t=t_i}-t^k\partial _t^kW_t^\mathcal L(f)(z)_{|t=t_{i+1}}\big)\\
&\hspace{-8cm}\quad -(t^k\partial _t^kW_t^\mathcal L(f)(y)_{|t=t_i}-t^k\partial _t^kW_t^\mathcal L(f)(y)_{|t=t_{i+1}})\big|^\sigma\Big)^{1/\sigma },\quad \mbox{ a.e. } x\in B(x_0,r_0).
\end{align*}
It follows that
\begin{align*}
\int_{B(x_0,r_0)}\Big(V_\sigma (\{t^k\partial _t^kW_t^\mathcal L\}_{t>0})(f)(x)-{\rm ess}\hspace{-2mm}\inf_{\hspace{-5mm}y\in B(x_0,r_0)}V_\sigma (\{t^k\partial _t^kW_t^\mathcal L\}_{t>0})(f)(y)\Big)dx\\
&\hspace{-11cm}\leq \int_{B(x_0,r_0)}V_\sigma (\{t^k\partial _t^kW_t^\mathcal L\}_{t\in (0,8r_0^2]})(f)(x)dx\\
&\hspace{-11cm}\quad +|B(x_0,r_0)|\;\;{\rm ess}\hspace{-2mm}\sup_{\hspace{-5mm}z,y\in B(x_0,r_0)}\sup_{8r_0^2\leq t_n<...<t_1}\Big(
\sum_{i=1}^{n-1}\big|[t^k\partial _t^kW_t^\mathcal L(f)(z)_{|t=t_i}-t^k\partial _t^kW_t^\mathcal L(f)(z)_{|t=t_{i+1}}]\\
&\hspace{-11cm}\quad -[t^k\partial _t^kW_t^\mathcal L(f)(y)_{|t=t_i}-t^k\partial _t^kW_t^\mathcal L(f)(y)_{|t=t_{i+1}}]\big|^\sigma\Big)^{1/\sigma }\\
&\hspace{-11cm}=:G_1(f)+G_2(f).
\end{align*}
We now estimate $G_1(f)$ and $G_2(f)$ separately. Firstly we consider $G_1(f)$. The function $f$ is decomposed as follows:
$$
f=(f-f_{B(x_0,r_0)})\mathcal{X}_{B(x_0,2r_0)}+(f-f_{B(x_0,r_0)})\mathcal{X}_{B(x_0,2r_0)^c}+f_{B(x_0,r_0)}=:f_1+f_2+f_3.
$$
It is clear that $G_1(f)\leq \sum_{j=1}^3G_1(f_j)$. Also, by Proposition \ref{Prop2.3}, (a), $w^{-\frac{1}{p-1}}\in A_{p'}^{\rho ,\theta}(\mathbb R^d)$. Then, $V_\sigma(\{t^k\partial_t^kW_t^\mathcal{L}\}_{t>0})$ is bounded from $L^{p'}(\mathbb{R}^d,w^{-\frac{1}{p-1}})$ into itself. It follows that
\begin{align}
G_1(f_1)&\leq w(B(x_0,r_0))^{1/p}\left(\int_{\mathbb R^d}|V_\sigma (\{t^k\partial _t^kW_t^\mathcal{L}\}_{t>0})(f_1)(x)|^{p'}w^{-\frac{1}{p-1}}(x)dx\right)^{1/p'}\nonumber\\
&\leq Cw(B(x_0,r_0))^{1/p}\left(\int_{B(x_0,2r_0)}|f(x)-f_{B(x_0,r_0)}|^{p'}w^{-\frac{1}{p-1}}(x)dx\right)^{1/p'}\nonumber\\
&\leq Cw(B(x_0,r_0))^{1/p}\left(\Big(\int_{B(x_0,2r_0)}|f(x)-f_{B(x_0,2r_0)}|^{p'}w^{-\frac{1}{p-1}}(x)dx\Big)^{1/p'}\right.\nonumber\\
&\quad +\left.|f_{B(x_0,2r_0)}-f_{B(x_0,r_0)}|\Big(\int_{B(x_0,2r_0)}w(x)^{-\frac{1}{p-1}}dx\Big)^{1/p'}\right)\nonumber\\
&\leq Cw(B(x_0,r_0))^{1/p}\Big(w(B(x_0,2r_0))^{1/p'}|B(x_0,2r_0)|^\alpha \|f\|_{{\rm BMO}_{\mathcal L,w}^\alpha (\mathbb R^d)} \label{D1}\\
&\quad +\frac{1}{w(B(x_0,2r_0))^{1/p}}\int_{B(x_0,2r_0)}|f(x)-f_{B(x_0,2r_0)}|dx\Big)\nonumber\\
&\leq Cw(B(x_0,r_0))^{1/p}w(B(x_0,2r_0)^{1/p'}|B(x_0,2r_0)|^\alpha \|f\|_{{\rm BMO}_{\mathcal L,w}^\alpha (\mathbb R^d)}\nonumber\\[0.2cm]
&\leq C|B(x_0,r_0)|^\alpha w(B(x_0,r_0))\|f\|_{{\rm BMO}_{\mathcal L,w}^\alpha (\mathbb R^d)}.\label{D2}
\end{align}
In \eqref{D1} we use estimate \eqref{2.1} and that $w\in A_p^{\rho ,\theta}(\mathbb R^d)$. In \eqref{D2} we have taken into account Proposition \ref{Prop2.3}, (c).

To analyze $G_1(f_2)$ we write 
\begin{align*}
G_1(f_2)&=\int_{B(x_0,r_0)}\sup_{0<t_n<...<t_1\leq 8r_0^2}\Big(\sum_{i=1}^{n-1}\Big|\int_{t_{i+1}}^{t_i}\partial _t(t^k\partial _t^kW_t^\mathcal L(f_2)(x))dt\Big|^\sigma\Big)^{1/\sigma}dx\\
&\leq \int_{B(x_0,r_0)}\int_0^{8r_0^2}\big|\partial _t(t^k\partial _t^kW_t^\mathcal L(f_2)(x))\big|dtdx.
\end{align*}
According to Proposition \ref{Prop2.1}, (a), we have that
\begin{align*}
G_1(f_2)&\leq C\int_{B(x_0,r_0)}\int_{\mathbb R^d\setminus B(x_0,2r_0)}|f(y)-f_{B(x_0,r_0)}|\int_0^{8r_0^2}e^{-c\frac{|x-y|^2}{t}}t^{-\frac{d}{2}-1}dtdydx\\
&\leq C\int_{B(x_0,r_0)}\int_{\mathbb R^d\setminus B(x_0,2r_0)}|f(y)-f_{B(x_0,r_0)}|\frac{e^{-c\frac{|x-y|^2}{r_0^2}}}{|x-y|^d}dydx\\
&\leq C\int_{B(x_0,r_0)}\int_{\mathbb R^d\setminus B(x_0,2r_0)}|f(y)-f_{B(x_0,r_0)}|\frac{e^{-c\frac{|x_0-y|^2}{r_0^2}}}{|x_0-y|^d}dydx\\
&\leq C|B(x_0,r_0)|\sum_{i=1}^\infty \frac{e^{-c2^{2i}}}{(2^ir_0)^d}\int_{B(x_0,2^{i+1}r_0)\setminus B(x_0,2^ir_0)}|f(y)-f_{B(x_0,r_0)}|dy\\
&\leq C\sum_{i=1}^\infty \frac{e^{-c2^{2i}}}{2^{id}}\int_{B(x_0,2^{i+1}r_0)}|f(y)-f_{B(x_0,r_0)}|dy\\
&\leq C\sum_{i=1}^\infty \frac{e^{-c2^{2i}}}{2^{id}}\Big(\int_{B(x_0,2^{i+1}r_0)}|f(y)-f_{B(x_0,2^{i+1}r_0)}|dy\\
&\quad +|B(x_0,2^{i+1}r_0)|\sum_{j=0}^i|f_{B(x_0,2^{j+1}r_0)}-f_{B(x_0,2^jr_0)}|\Big).
\end{align*}
We now observe that, for every $n\in \mathbb{N}$, according to Proposition \ref{Prop2.3}, (c),
\begin{align}\label{difint}
\int_{B(x_0,2^nr_0)}|f(y)-f_{B(x_0,2^nr_0)}|dy&\leq C|B(x_0,2^nr_0)|^\alpha w(B(x_0,2^nr_0))\|f\|_{{\rm BMO}_{\mathcal L,w}^\alpha (\mathbb R^d)}\nonumber\\
&\leq C2^{n(d(p +\alpha)+p\theta)}|B(x_0,r_0)|^\alpha w(B(x_0,r_0))\|f\|_{{\rm BMO}_{\mathcal L,w}^\alpha (\mathbb R^d)},
\end{align}
and
\begin{align}\label{dif}
|f_{B(x_0,2^{n+1}r_0)}-f_{B(x_0,2^nr_0)}|&\leq \frac{1}{|B(x_0,2^nr_0)|}\int_{B(x_0,2^{n+1}r_0)}|f(y)-f_{B(x_0,2^{n+1}r_0)}|dy\nonumber\\
&\leq C2^{n(d(p+\alpha -1)+p\theta)}|B(x_0,r_0)|^{\alpha -1}w(B(x_0,r_0))\|f\|_{{\rm BMO}_{\mathcal L,w}^\alpha (\mathbb R^d)}.
\end{align}
Thus,
\begin{align*}
G_1(f_2)&\leq C|B(x_0,r_0)|^\alpha w(B(x_0,r_0))\|f\|_{{\rm BMO}_{\mathcal L,w}^\alpha (\mathbb R^d})\sum_{i=1}^\infty e^{-c2^{2i}}\Big(2^{i(d(p+\alpha -1)+p\theta)}+\sum_{j=0}^i2^{j(d(p+\alpha -1)+p\theta)}\Big)\\
&\leq C|B(x_0,r_0)|^\alpha w(B(x_0,r_0))\|f\|_{{\rm BMO}_{\mathcal L,w}^\alpha (\mathbb R^d}).
\end{align*}

Let us deal now with $G_1(f_3)$. By $\{W_t\}_{t>0}$ we denote the classical heat semigroup, that is, for every $t>0$,
$$
W_t(f)=\int_{\mathbb{R}^d}W_t(x-y)f(y)dy,\,\,\,x\in \mathbb{R}^d,
$$
where
$$
W_t(z)=\frac{1}{(4\pi t)^{d/2}}e^{-|z|^2/4t}, \,\,\,z\in \mathbb{R}^d.
$$
By taking into account that $\partial _t(t^k\partial _t^kW_t(1)(x))=0$, $x\in \mathbb R^d$ and $t>0$, we can write
\begin{align*}
G_1(f_3)&\leq |f_{B(x_0,r_0)}|\int_{B(x_0,r_0)}\int_0^{8r_0^2}\Big|\int_{\mathbb R^d}\partial _t(t^k\partial _t^kW_t^{\mathcal L}(x,y))dy\Big|dtdx\\
&=|f_{B(x_0,r_0)}|\int_{B(x_0,r_0)}\int_0^{8r_0^2}\Big|\int_{\mathbb R^d}\partial _t(t^k\partial _t^kW_t^{\mathcal L}(x,y)-t^k\partial _t^kW_t(x-y))dy\Big|dtdx\\
&\leq |f_{B(x_0,r_0)}|\\
&\quad \times \int_{B(x_0,r_0)}\int_0^{8r_0^2}\Big(\int_{|x-y|\leq \rho (x_0)}+\int_{|x-y|\geq \rho (x_0)}\Big)|\partial _t(t^k\partial _t^kW_t^{\mathcal L}(x,y)-t^k\partial _t^kW_t(x-y))|dydtdx\\
&=:G_{11}(f_3)+G_{12}(f_3).
\end{align*}

According to Proposition \ref{Prop2.1}(d) we have that
\begin{align*}
G_{11}(f_3)&\leq C|f_{B(x_0,r_0)}|\int_{B(x_0,r_0)}\int_0^{8r_0^2}\int_{|x-y|\leq \rho (x_0)}\Big(\frac{\sqrt{t}}{\rho (x)}\Big)^{2-\frac{d}{q}}\frac{e^{-c\frac{|x-y|^2}{t}}}{t^{\frac{d}{2}+1}}dydtdx\\
&\leq C|f_{B(x_0,r_0)}|\int_{B(x_0,r_0)}\int_{|x-y|\leq \rho (x_0)}\frac{e^{-c\frac{|x-y|^2}{r_0^2}}}{\rho (x)^{2-\frac{d}{q}}}\int_0^{8r_0^2}\frac{e^{-c\frac{|x-y|^2}{t}}}{t^{\frac{d}{2}+\frac{d}{2q}}}dtdydx\\
&\leq C|f_{B(x_0,r_0)}|\rho (x_0)^{\frac{d}{q}-2}\int_{B(x_0,r_0)}\int_{|x-y|\leq \rho (x_0)}\frac{e^{-c\frac{|x-y|^2}{r_0^2}}}{|x-y|^{d+\frac{d}{q}-2}}dydx\\
&\leq C|f_{B(x_0,r_0)}|\rho (x_0)^{\frac{d}{q}-2}|B(x_0,r_0)|\int_0^{\rho (x_0)}e^{-c\frac{s^2}{r_0^2}}s^{1-\frac{d}{q}}ds\\
&\leq C|f_{B(x_0,r_0)}|\rho (x_0)^{\frac{d}{q}-2}|B(x_0,r_0)|\int_0^{\infty}e^{-c\frac{s^2}{r_0^2}}s^{1-\frac{d}{q}}ds\\
&=C|f_{B(x_0,r_0)}||B(x_0,r_0)|\Big(\frac{r_0}{\rho (x_0)}\Big)^{2-\frac{d}{q}}.
\end{align*}
In the third inequality we have taken into account that $\rho (x)\sim \rho (x_0)$ provided that $|x-x_0|<\rho (x_0)$.

On the other hand, by Proposition \ref{Prop2.1}, (a), and since 
\begin{equation}\label{dergaus}
|t^k\partial _t^kW_t(z)|\le \frac{C}{t^{d/2}}e^{-c|z|^2/t},\,\,\,z\in \mathbb{R}^d \,\,\,\hbox{and}\,\,\,t>0,
\end{equation}
we have that
\begin{align*}
G_{12}(f_3)&\leq C|f_{B(x_0,r_0)}|\int_{B(x_0,r_0)}\int_0^{8r_0^2}\int_{|x-y|\geq \rho (x_0)}\frac{e^{-c\frac{|x-y|^2}{t}}}{t^{\frac{d}{2}+1}}dydtdx\\
&\leq C|f_{B(x_0,r_0)}|\int_{B(x_0,r_0)}\int_{|x-y|\geq \rho (x_0)}e^{-c\frac{|x-y|^2}{r_0^2}}\int_0^{8r_0^2}\frac{e^{-c\frac{|x-y|^2}{t}}}{t^{\frac{d}{2}+1}}dtdydx\\
&\leq C|f_{B(x_0,r_0)}|\int_{B(x_0,r_0)}\int_{|x-y|\geq \rho (x_0)}\frac{e^{-c\frac{|x-y|^2}{r_0^2}}}{|x-y|^d}dydx\\
&\leq C|f_{B(x_0,r_0)}||B(x_0,r_0)|\int_{\rho (x_0)}^\infty \frac{e^{-c\frac{s^2}{r_0^2}}}{s}ds\leq C|f_{B(x_0,r_0)}||B(x_0,r_0)|\Big(\frac{r_0}{\rho (x_0)}\Big)^\beta,
\end{align*}
provided that $\beta >0$.

We deduce that, for $\beta >0$,
\begin{equation}\label{G1f3}
G_1(f_3)\leq C|f_{B(x_0,r_0)}||B(x_0,r_0)|\left(\Big(\frac{r_0}{\rho (x_0)}\Big)^{2-\frac{d}{q}}+\Big(\frac{r_0}{\rho (x_0)}\Big)^\beta\right).
\end{equation}
We now choose $i_0\in \mathbb{N}$ such that $2^{i_0}r_0<\rho (x_0)\leq 2^{i_0+1}r_0$. By \eqref{dif} we get
\begin{align}\label{promedioB}
|f_{B(x_0,r_0)}|&\leq \sum_{i=0}^{i_0}|f_{B(x_0,2^{i+1}r_0)}-f_{B(x_0,2^ir_0)}|+|f_{B(x_0,2^{i_0+1}r_0)}|\nonumber\\
&\leq C|B(x_0,r_0)|^{\alpha -1}w(B(x_0,r_0))\|f\|_{{\rm BMO}_{\mathcal L,w}^\alpha (\mathbb R^d)}\sum_{i=0}^{i_0+1}2^{i(d(p+\alpha -1)+p\theta)}\nonumber\\
&\leq C2^{i_0(d(p+\alpha -1)+p\theta)}|B(x_0,r_0)|^{\alpha -1}w(B(x_0,r_0))\|f\|_{{\rm BMO}_{\mathcal L,w}^\alpha (\mathbb R^d)}\nonumber\\
&\leq C\Big(\frac{\rho (x_0)}{r_0}\Big)^{d(p+\alpha -1)+p\theta}|B(x_0,r_0)|^{\alpha -1}w(B(x_0,r_0))\|f\|_{{\rm BMO}_{\mathcal L,w}^\alpha (\mathbb R^d)}.
\end{align}
Since $2-\frac{d}{q}>d(p+\alpha -1)+p\theta$ and taking $\beta =d(p+\alpha -1)+p\theta$ in \eqref{G1f3} we obtain
$$
G_1(f_3)\leq C|B(x_0,r_0)|^\alpha w(B(x_0,r_0))\|f\|_{{\rm BMO}_{\mathcal L, w}^\alpha (\mathbb R^d)}.
$$
By putting together the above estimations we obtain
\begin{equation}\label{*1}
G_1(f)\leq C|B(x_0,r_0)|^\alpha w(B(x_0,r_0))\|f\|_{{\rm BMO}_{\mathcal L, w}^\alpha (\mathbb R^d)}.
\end{equation}

We now deal with $G_2(f)$. We can write
\begin{align*}
G_2(f)&\leq |B(x_0,r_0)|\;{\rm ess}\hspace{-2mm}\sup_{\hspace{-5mm}x,y\in B(x_0,r_0)}\int_{8r_0^2}^\infty \Big|\int_{\mathbb R^d}\partial _t(t^k\partial _t^kW_t^{\mathcal L}(x,z)-t^k\partial _t^kW_t^{\mathcal L}(y,z))f(z)dz\Big|dt\\
&\leq |B(x_0,r_0)|\\
&\quad \times\;{\rm ess}\hspace{-2mm}\sup_{\hspace{-5mm}x,y\in B(x_0,r_0)}\left(\int_{8\rho(x_0)^2}^\infty +\int_{8r_0^2}^{8\rho(x_0)^2}\right)\Big|\int_{\mathbb R^d}\partial _t(t^k\partial _t^kW_t^{\mathcal L}(x,z)-t^k\partial _t^kW_t^{\mathcal L}(y,z))f(z)dz\Big|dt\\
&=: G_{21}(f)+G_{22}(f).
\end{align*}

We firstly estimate $G_{21}(f)$. According to Proposition \ref{Prop2.1}, (b), we deduce that, for every $0<\delta <\delta_0$, there exists $C>0$ such that, for each $x,y\in B(x_0,r_0)$ and $t>\rho(x_0)^2$,
\begin{align*}
\Big|\int_{\mathbb R^d}\partial _t(t^k\partial _t^kW_t^{\mathcal L}(x,z)-t^k\partial_t^kW_t^{\mathcal L}(y,z))f(z)dz\Big|&\leq C\int_{\mathbb R^d} \Big(\frac{|x-y|}{\sqrt{t}}\Big)^\delta\frac{e^{-c\frac{|y-z|^2}{t}}}{t^{\frac{d}{2}+1}}|f(z)|dz\\
&\hspace{-7cm}\leq \frac{C}{t^{\frac{d}{2}+1}}\Big(\frac{|x-y|}{\sqrt{t}}\Big)^\delta\left(\sum_{j=0}^\infty e^{-c2^{2j}}\int_{2^{j}\sqrt{t}\leq |y-z|<2^{j+1}\sqrt{t}}|f(z)|dz+\int_{|y-z|<2^{-1}\sqrt{t}}|f(z)|dz\right)\\
&\hspace{-7cm}\leq \frac{C}{t^{\frac{d}{2}+1}}\Big(\frac{|x-y|}{\sqrt{t}}\Big)^\delta\left(\sum_{j=0}^\infty e^{-c2^{2j}}\int_{|x_0-z|<2^{j+1}\sqrt{t}}|f(z)|dz+\int_{|x_0-z|<\sqrt{t}}|f(z)|dz\right)\\
&\hspace{-7cm}\leq \frac{C}{t^{\frac{d}{2}+1}}\Big(\frac{r_0}{\sqrt{t}}\Big)^\delta\|f\|_{{\rm BMO}_{\mathcal L,w}^\alpha (\mathbb R^d)}\sum_{j=0}^\infty e^{-c2^{2j}}|B(x_0, 2^j\sqrt{t})|^\alpha w(B(x_0,2^j\sqrt{t}))\\
&\hspace{-7cm}\leq \frac{C}{t}\Big(\frac{r_0}{\sqrt{t}}\Big)^\delta\|f\|_{{\rm BMO}_{\mathcal L,w}^\alpha (\mathbb R^d)}w(B(x_0,r_0))\frac{t^{\frac{d}{2}(p+\alpha-1)+\frac{p\theta}{2}}}{r_0^{p(\theta+d)}}.
\end{align*}

In the last inequality we have used Proposition \ref{Prop2.3}, (c). For every $x,y\in B(x_0,r_0)$ and $t>8\rho (x_0)^2$, 
\begin{align*}
\Big|\int_{\mathbb R^d}\partial _t(t^k\partial _t^k[W_t^{\mathcal L}(x,z)&-t^kW_t^{\mathcal L}(y,z)])f(z)dz\Big|\\
&\leq Cr_0^\delta w(B(x_0,r_0))|B(x_0,r_0)|^{\alpha-1}\|f\|_{{\rm BMO}_{\mathcal L,w}^\alpha (\mathbb R^d)}\Big(\frac{t}{r_0}\Big)^{\frac{d}{2}(p+\alpha-1)+\frac{p\theta}{2}-\frac{\delta}{2}}\frac{1}{t}.
\end{align*}
It follows that
\begin{align*}
G_{21}(f)&\leq C|B(x_0,r_0)|^\alpha w(B(x_0,r_0))\|f\|_{{\rm BMO}_{\mathcal L, w}^\alpha (\mathbb R^d)}r_0^{\delta -d(p+\alpha -1)-p\theta}\int_{8\rho (x_0)^2}^\infty t^{\frac{d}{2}(p+\alpha -1)-\frac{\delta-p\theta }{2}-1}dt\\
&\leq C|B(x_0,r_0)|^\alpha w(B(x_0,r_0))\|f\|_{{\rm BMO}_{\mathcal L, w}^\alpha (\mathbb R^d)}\Big(\frac{r_0}{\rho (x_0)}\Big)^{\delta-p\theta -d(p+\alpha -1)}\\
&\leq C|B(x_0,r_0)|^\alpha w(B(x_0,r_0))\|f\|_{{\rm BMO}_{\mathcal L, w}^\alpha (\mathbb R^d)},
\end{align*}
provided that $\delta >d(p+\alpha -1)+p\theta$. Note that we can choose this $\delta$ because $\delta_0>d(p+\alpha -1)+p\theta$.

To deal with $G_{22}(f)$ we write, for every $t\in (8r_0^2,8\rho (x_0)^2)$ and $x,y\in B(x_0,r_0)$,
\begin{align*}
\Big|\int_{\mathbb R^d}\partial _t(t^k\partial _t^k[W_t^{\mathcal L}(x,z)-W_t^{\mathcal L}(y,z)])f(z)dz\Big|&\leq \Big|\int_{\mathbb R^d}\partial _t(t^k\partial _t^k[W_t^{\mathcal L}(x,z)-W_t^{\mathcal L}(y,z)])(f(z)-f_{B(x_0,r_0)})dz\Big|\\
&\hspace{-6cm}\quad +|f_{B(x_0,r_0)}|\Big|\int_{\mathbb R^d}\partial _t(t^k\partial _t^k[W_t^{\mathcal L}(x,z)-W_t^{\mathcal L}(y,z)])dz\Big|=: F_1(x,y,t)+F_2(x,y,t).
\end{align*}
Thus,
$$
G_{22}(f)\leq C|B(x_0,r_0)|\sup _{x,y\in B(x_0,r_0)}\int_{8r_0^2}^{8\rho (x_0)^2}(F_1(x,y,t)+F_2(x,y,t))dt.
$$
By using again Proposition \ref{Prop2.1}, (b), for every $0<\delta <\delta_0$, and proceeding as above we have that, for every $x,y\in B(x_0,r_0)$ and $8r_0^2<t<8\rho(x_0)^2$,
\begin{align*}
F_1(x,y,t)&\leq C\int_{\mathbb R^d}\Big(\frac{|x-y|}{\sqrt{t}}\Big)^\delta \frac{e^{-c\frac{|y-z|^2}{t}}}{t^{\frac{d}{2}+1}}|f(z)-f_{B(x_0,r_0)}|dz\\
&\hspace{-1cm}\leq \frac{C}{t^{\frac{d}{2}+1}}\Big(\frac{r_0}{\sqrt{t}}\Big)^\delta\\
&\hspace{-1cm}\quad \times \Big(\sum_{j=0}^\infty e^{-c\frac{2^{2j}r_0^2}{t}}\int_{2^{j}r_0\leq |y-z|<2^{j+1}r_0}|f(z)-f_{B(x_0,r_0)}|dz+\int_{|y-z|<2^{-1}r_0}|f(z)-f_{B(x_0,r_0)}|dz\Big)\\
&\hspace{-1cm}\leq \frac{C}{t^{\frac{d}{2}+1}}\Big(\frac{r_0}{\sqrt{t}}\Big)^\delta\\
&\hspace{-1cm}\quad \times \Big(\sum_{j=0}^\infty e^{-c\frac{2^{2j}r_0^2}{t}}\int_{B(x_0,2^{j+1}r_0)}|f(z)-f_{B(x_0,r_0)}|dz+\int_{B(x_0,r_0)}|f(z)-f_{B(x_0,r_0)}|dz\Big)\\
&\hspace{-1cm}\leq \frac{C}{t^{\frac{d}{2}+1}}\Big(\frac{r_0}{\sqrt{t}}\Big)^\delta\\
&\hspace{-1cm}\quad \times \Big(\sum_{j=0}^\infty e^{-c\frac{2^{2j}r_0^2}{t}}\Big[\int_{B(x_0,2^{j+1}r_0)}|f(z)-f_{B(x_0,2^{j+1}r_0)}|dz\\
&\hspace{-1cm}\quad \quad +|B(x_0,2^{j+1}r_0)|\sum_{i=0}^j|f_{B(x_0,2^{i+1}r_0)}-f_{B(x_0,2^ir_0)}|\Big]+\int_{B(x_0,r_0)}|f(z)-f_{B(x_0,r_0)}|dz\Big).
\end{align*}
Now, according to \eqref{difint} and \eqref{dif} we obtain
\begin{align*}
F_1(x,y,t)&\leq \frac{C}{t^{\frac{d}{2}+1}}\Big(\frac{r_0}{\sqrt{t}}\Big)^\delta|B(x_0,r_0)|^\alpha w(B(x_0,r_0)\|f\|_{{\rm BMO}_{\mathcal L,w}^\alpha (\mathbb R^d)}\\
&\quad \times \Big(\sum_{j=0}^\infty e^{-c\frac{2^{2j}r_0^2}{t}}(j+1)2^{j(d(p+\alpha )+p\theta)}+1\Big),\quad x,y\in B(x_0,r_0),\;8r_0^2<t<8\rho(x_0)^2.
\end{align*}
It follows that
\begin{align*}
\int_{8r_0^2}^{8\rho (x_0)^2}F_1(x,y,t)dt&\leq Cr_0^\delta|B(x_0,r_0)|^\alpha w(B(x_0,r_0))\|f\|_{{\rm BMO}_{\mathcal L,w}^\alpha (\mathbb R^d)}\\
&\hspace{-2cm}\quad \times \left(\sum_{j=0}^\infty (j+1)2^{j(d(p+\alpha )+p\theta)}\int_{8r_0^2}^{8\rho (x_0)^2}\frac{e^{-c\frac{2^{2j}r_0^2}{t}}}{t^{\frac{d+\delta}{2}+1}}dt+\int_{8r_0^2}^{8\rho (x_0)^2}\frac{dt}{t^{\frac{d+\delta}{2}+1}}\right)\\
&\hspace{-2cm}\leq Cr_0^\delta |B(x_0,r_0)|^\alpha w(B(x_0,r_0))\|f\|_{{\rm BMO}_{\mathcal L,w}^\alpha (\mathbb R^d)}\left(\sum_{j=0}^\infty \frac{(j+1)2^{j(d(p+\alpha )+p\theta)}}{(2^jr_0)^{d+\delta }}+\frac{1}{r_0^{d+\delta }}\right)\\
&\hspace{-2cm}\leq C|B(x_0,r_0)|^{\alpha-1} w(B(x_0,r_0))\|f\|_{{\rm BMO}_{\mathcal L,w}^\alpha (\mathbb R^d)},\quad x,y\in B(x_0,r_0), 
\end{align*}
provided that $\delta >d(p+\alpha -1)+p\theta$.

Finally, let $m\in \mathbb{N}$. By Proposition \ref{Prop2.1}, (c), there exists $C>0$ such that
$$
\left|\int_{\mathbb R^d}t^m\partial _t^{m+1}W_t^{\mathcal L}(x,y)dy\right|\leq \frac{C}{t}\Big(\frac{\sqrt{t}}{\rho (x)}\Big)^{\delta_0},\quad t\leq 8\rho (x)^2\mbox{ and }x\in \mathbb R^d, 
$$
and by, \cite[p. 98]{YYZ2}, for every $0<\delta<\delta_0$, there exists $C>0$ such that, for every $x,y\in B(x_0,r_0)$ and $t>8r_0^2$,
$$
|t^m\partial _t^{m+1}(W_t^{\mathcal L}(1)(x)-W_t^{\mathcal L}(1)(y))|\leq \frac{C}{t}\Big(\frac{r_0}{\sqrt{t}}\Big)^\delta.
$$
By using these estimates we can write, for every $x,y\in B(x_0,r_0)$ and $t\in [8r_0^2,8\rho (x_0)^2]$,
\begin{align*}
F_2(x,y,t)&\leq C|f_{B(x_0,r_0)}|\sum_{m=k-1}^k\Big[\Big|\int_{\mathbb R^d}t^m\partial _t^{m+1}W_t^{\mathcal L}(x,z)dz\Big|+\Big|\int_{\mathbb R^d}t^m\partial _t^{m+1}W_t^{\mathcal L}(y,z)])dz\Big|\Big]^{1/2}\\
&\quad \times |t^m\partial _t^{m+1}(W_t^{\mathcal L}(1)(x)-W_t^{\mathcal L}(1)(y))|^{1/2}\\ 
&\leq C|f_{B(x_0,r_0)}|\frac{1}{t}\Big(\frac{r_0}{\rho (x_0)}\Big)^{\delta/2},
\end{align*}
with $0<\delta<\delta_0$. By taking into account \eqref{promedioB} it follows that
\begin{align*}
\int_{8r_0^2}^{8\rho (x_0)^2}F_2(x,y,t)dt&\leq C|B(x_0,r_0)|^{\alpha -1}w(B(x_0,r_0))\|f\|_{{\rm BMO}_{\mathcal L,w}^\alpha (\mathbb R^d)}\Big(\frac{\rho (x_0)}{r_0}\Big)^{d(p+\alpha -1)+p\theta-\frac{\delta}{2}}\\
&\quad \times \int_{8r_0^2}^{8\rho (x_0)^2}\frac{dt}{t}\\
&\leq Cw(B(x_0,r_0))|B(x_0,r_0)|^{d(\alpha -1)}\|f\|_{{\rm BMO}_{\mathcal L,w}^\alpha (\mathbb R ^d)},\quad x,y\in B(x_0,r_0),
\end{align*}
provided that $\delta_0>\delta>2(d(p+\alpha -1)+p\theta)$. 

We conclude that
$$
G_{22}(f)\leq C|B(x_0,r_0)|^\alpha w(B(x_0,r_0))\|f\|_{{\rm BMO}_{\mathcal L,w}^\alpha (\mathbb R^d)}.
$$
We get 
\begin{equation}\label{*2}
G_2(f)\leq C|B(x_0,r_0)|^\alpha w(B(x_0,r_0))\|f\|_{{\rm BMO}_{\mathcal L,w}^\alpha (\mathbb R^d)}.
\end{equation}
Thus, by considering \eqref{*1} and \eqref{*2} the proof can be finished.
\section{Proof of Theorem \ref{Th1.2} for the oscillation operator $O(\{t^k\partial_t^kW_t^\mathcal{L}\}_{t>0},\{t_j\}_{j\in \mathbb{Z}})$}\label{S4}

In order to prove that the operator $O(\{t^k\partial_t^kW_t^\mathcal{L}\}_{t>0},\{t_j\}_{j\in \mathbb{Z}})$ is bounded from $L^p(\mathbb R^d,w)$ into itself for every $1<p<\infty$ and $w\in A_{p}^{\rho ,\infty}(\mathbb{R}^d)$, we can proceed as in the proof of \cite[Theorem 4]{TZ} and in \cite[Theorem 1.1]{BFHR1}. We sketch the main steps of the proof.

We firstly establish the result in the unweighted case, that is, we prove that $O(\{t^k\partial_t^kW_t^\mathcal{L}\}_{t>0},\{t_j\}_{j\in \mathbb{Z}})$ is bounded from $L^p(\mathbb R^d)$ into itself for every $1<p<\infty$. As far as we know a $L^p$-boundedness result for oscillation operators like the one established in \cite[Corollary 4.5]{LeMX} has not been proved. Since $\{W_t^{\mathcal L}\}_{t>0}$ is not Markovian, the $L^p$-boundedness of $O(\{W_t^\mathcal{L}\}_{t>0},\{t_j\}_{j\in \mathbb{Z}})$ can not be deduced from \cite[Theorem 3.3, (2)]{JR}.

Suppose that $F:(0,\infty )\longrightarrow \mathbb{C}$ is a derivable function. We have that
\begin{align}\label{E1}
O(\{F(t)\}_{t>0},\{t_j\}_{j\in \mathbb{Z}})&=\left(\sum_{i=-\infty }^{+\infty}\sup_{t_i\leq \varepsilon_i<\varepsilon _{i+1}\leq t_{i+1}}|F(\varepsilon _i)-F(\varepsilon _{i+1})|^2\right)^{1/2}\nonumber\\
&=\left(\sum_{i=-\infty }^{+\infty}\sup_{t_i\leq \varepsilon_i<\varepsilon _{i+1}\leq t_{i+1}}\Big|\int_{\varepsilon _i}^{\varepsilon _{i+1}}F'(t)dt\Big|^2\right)^{1/2}\nonumber\\
&\leq \left(\sum_{i=-\infty }^{+\infty}\sup_{t_i\leq \varepsilon_i<\varepsilon _{i+1}\leq t_{i+1}}\Big(\int_{\varepsilon _i}^{\varepsilon _{i+1}}|F'(t)|dt\Big)^2\right)^{1/2}\nonumber\\
&\leq C\sum_{i=-\infty }^{+\infty}\int_{t_i}^{t_{i+1}}|F'(t)|dt\leq C\int_0^\infty |F'(t)|dt.
\end{align}
This inequality plays an important role in our proof for the boundedness properties for \newline $O(\{t^k\partial _t^kW_t^\mathcal{L}\}_{t>0},\{t_j\}_{j\in \mathbb{Z}})$.

It is easy to see that if $F$ is a complex function defined in $(0,\infty )$ and $O(\{F(t)\}_{t>0},\{t_j\}_{j\in \mathbb{Z}})=0$, then $F$ is constant. The oscillation operator associated with $\{t_j\}_{j\in \mathbb Z}$ defines a seminorm in the space $\mathcal{F}$ of complex functions defined in $(0,\infty )$ such that $O(\{F(t)\}_{t>0},\{t_j\}_{j\in \mathbb{Z}})<\infty$.

We consider the quotient space $\mathcal{F}/\!\sim$ where $\sim$ is the binary relation defined as follows: if $F_1, F_2\in \mathcal F$ we say that $F_1\sim F_2$ when $F_1-F_2$ is constant. The oscillation defines a norm on $\mathcal{F}/\!\sim$ and $(\mathcal{F}/\!\sim , O(\cdot ,\{t_j\}_{j\in \mathbb{Z}}))$ is a Banach space. To see the oscillation as a norm allows us to simplify our arguments. We can also understand our oscillation operators $O(\{t^k\partial _t^kW_t^\mathcal{L}\}_{t>0},\{t_j\}_{j\in \mathbb{Z}})$ as Banach valued singular integral operators.

In order to prove that $O(\{t^k\partial _t^kW_t^\mathcal{L}\}_{t>0},\{t_j\}_{j\in \mathbb{Z}})$ defines a bounded operator from $L^p(\mathbb R^d)$ into itself, $1<p<\infty$, we exploit that the Schr\"odinger operator $\mathcal L$ is a nice (in some sense) perturbation of the Euclidean Laplacian. 

We now explain the procedure (see \cite{BFHR1}).

We split the region $\mathbb R^d\times \mathbb R^d$ in two parts:
$$
L=\{(x,y)\in \mathbb R^d\times \mathbb R^d: |x-y|<\rho (x)\}$$
and $G=(\mathbb R^d\times \mathbb R^d)\setminus L$. $L$ and $G$ mean local and global regions, respectively. To simplify we write $T_\mathcal{L}=O(\{t^k\partial _t^kW_t^\mathcal{L}\}_{t>0},\{t_j\}_{j\in \mathbb{Z}})$. 

We decompose the operator $T_\mathcal{L}$ in two parts: the local part 
$T_{\mathcal L,{\rm loc}}(f)(x)=T_\mathcal{L}(f\mathcal X_{B(x,\rho (x))})(x)$, $x\in\mathbb R^d$, and the global one, $T_{\mathcal{L},{\rm glob}}=T_\mathcal L-T_{\mathcal L,{\rm loc}}$.

We define the operators $T_{-\Delta}$, $T_{-\Delta ,{\rm loc}}$ and $T_{-\Delta ,{\rm glob}}$ as above by replacing the Schr\"odinger operator by the Euclidean Laplacian.

We decompose the operator $T_\mathcal{L}$ as follows:
$$
T_\mathcal L=(T_{\mathcal L,{\rm loc}}-T_{-\Delta ,{\rm loc}})+T_{-\Delta ,{\rm loc}}+T_{\mathcal L,{\rm glob}}.
$$
Our objective is to establish that the operators $T_{\mathcal L,{\rm loc}}-T_{-\Delta ,{\rm loc}}$, $T_{-\Delta,{\rm loc}}$ and $T_{\mathcal L,{\rm glob}}$ are bounded from $L^p(\mathbb R^ d)$ into itself, for every $1<p<\infty$.

We first study $T_{-\Delta ,{\rm loc}}$. We consider the function $\phi (z)=e^{-z}$, $z\in (0,\infty )$. The Euclidean heat kernel in $\mathbb{R}^d$ is defined by
$$
W_t(z)=\frac{1}{(4\pi t)^{d/2}}e^{-\frac{|z|^2}{4t}}=\frac{1}{(4\pi t)^{d/2}}\phi \Big(\frac{|z|^2}{4t}\Big),\quad z\in \mathbb R^d\mbox{ and }t>0.
$$
By using the Fa\`a di Bruno's formula we obtain
\begin{align*}
\partial _t^kW_t(z)&=\sum_{j=0}^kc_jt^{-\frac{d}{2}-(k-j)}\partial _t^j\phi \Big(\frac{|z|^2}{4t}\Big)\\
&=\sum_{j=0}^kc_jt^{-\frac{d}{2}-(k-j)}\sum_{m_1+2m_2+...+jm_j=j}d_{m_1,...,m_j}^j\phi \Big(\frac{|z|^2}{4t}\Big)\frac{|z|^{2(m_1+...+m_j)}}{t^{m_1+...+m_j+j}}\\
&=\sum_{j=0}^k\sum_{m_1+2m_2+...+jm_j=j}c_jd_{m_1,...,m_j}^j\frac{1}{t^{\frac{d}{2}+k}}\phi \Big(\frac{|z|^2}{4t}\Big)\Big(\frac{|z|^2}{t}\Big)^{m_1+...+m_j},\quad z\in \mathbb R^d\mbox{ and }t>0,
\end{align*}
where $c_j$, $d_{m_1,...,m_j}^j\in \mathbb R$, $j=0,...,k$ y $m_1+2m_2+...+jm_j=j$, $m_1,...,m_j\in \mathbb{N}$. Then,
\begin{equation}\label{E2}
t^k\partial _t^kW_t(z)=\frac{1}{t^{d/2}}\psi \Big(\frac{|z|}{\sqrt{t}}\Big),\quad z\in \mathbb R^d\mbox{ and }t>0,
\end{equation}
being
$$
\psi (u)=\sum_{j=0}^k\sum_{m_1+2m_2+...+jm_j=j}c_jd_{m_1,...,m_j}^j\phi (u^2)u^{2(m_1+...+m_j)},\quad u\in \mathbb{R}.
$$
Note that $\psi$ is in the Schwartz class $\mathcal S(\mathbb R)$. According to \cite[Lemma 2.4, (1)]{CJRW1} the operator $T_{-\Delta }$ is bounded from $L^p(\mathbb R^d)$ into itself, for every $1<p<\infty$.

According to \cite[Proposition 5]{DGMTZ} we choose a sequence $\{x_j\}_{j\in \mathbb N}\subset \mathbb R^d$ such that by defining $Q_j=B(x_j,\rho (x_j))$ the following two properties holds:

(i) $\bigcup_{j\in \mathbb N}Q_j=\mathbb R^d$;

(ii) For every $m\in \mathbb N$ there exist $\gamma,\beta \in \mathbb{N}$ such that, for every $j\in \mathbb N$, the set 
$$
\{\ell \in \mathbb N: 2^mQ_\ell\cap 2^mQ_j\neq \emptyset\}
$$
has at most $\gamma 2^{m\beta}$ elements.

Let $j\in \mathbb{N}$. If $x\in Q_j$ and $z\in B(x,\rho (x))$, then
$$
|z-x_j|\leq |z-x|+|x-x_j|\leq \rho (x)+\rho (x_j)\leq C_1\rho (x_j),
$$
because $\rho (x)\sim \rho (x_j)$. Here $C_1$ does not depend on $j$.

We consider, for every $t>0$, the operator 
$$
H_t^j(f)(x)=\mathcal{X}_{Q_j}(x)\int_{B(x_j,C_1\rho (x_j))\setminus B(x,\rho (x))}t^k\partial _t^kW_t(x-y)f(y)dy,\quad x\in \mathbb R^d.
$$
By using \eqref{E1} and \eqref{E2} we deduce that
\begin{align*}
O(\{t^k\partial _t^kW_t(z)\}_{t>0},\{t_j\}_{j\in \mathbb{Z}})&\leq C\int_0^\infty \Big|\partial _t\Big(\frac{1}{t^{d/2}}\psi \Big(\frac{|z|}{\sqrt{t}}\Big)\Big)\Big|dt\\
&\leq C\int_0^\infty \frac{1}{t^{\frac{d}{2}+1}}\left(\Big|\psi  \Big(\frac{|z|}{\sqrt{t}}\Big)\Big|+\frac{|z|}{\sqrt{t}}\Big|\psi ' \Big(\frac{|z|}{\sqrt{t}}\Big)\Big|\right)dt\leq \frac{C}{|z|^d},\quad z\in \mathbb R^d\setminus \{0\}.
\end{align*}
It follows that
\begin{align*}
O(\{H_t^j\}_{t>0},\{t_j\}_{j\in \mathbb{Z}})(f)(x)&\\
&\hspace{-3cm}\leq \mathcal{X}_{Q_j}(x)\int_{B(x_j,C_1\rho (x_j))\setminus B(x,\rho (x))}O(\{t^k\partial _t^kW_t(x-y)\}_{t>0},\{t_j\}_{j\in \mathbb{Z}})|f(y)|dy\\
&\hspace{-3cm}\leq C\mathcal{X}_{Q_j}(x)\int_{B(x_j,C_1\rho (x_j))\setminus B(x,\rho (x))}\frac{1}{|x-y|^d}|f(y)|dy\leq \frac{C}{\rho (x)^d}\mathcal{X}_{Q_j}(x)\int_{B(x_j,C_1\rho (x_j))}|f(y)|dy\\
&\hspace{-3cm}\leq \frac{C}{\rho (x_j)^d}\mathcal{X}_{Q_j}(x)\int_{B(x_j,C_1\rho (x_j))}|f(y)|dy\leq C\mathcal{X}_{Q_j}(x)\mathcal M_{\rm HL}(f)(x),\quad x\in \mathbb R^d.
\end{align*}
Here $\mathcal M_{\rm HL}$ represents the classical Hardy-Littlewood maximal operator. We have that
$$
T_{-\Delta }(\mathcal {X}_{B(x_j,C_1\rho (x_j))}f)=T_{-\Delta ,{\rm loc}}(f)(x)+T_{-\Delta }(\mathcal X_{B(x_j,C_1\rho (x_j))\setminus B(x,\rho (x))}f)(x),\quad x\in Q_j.
$$
Then,
$$
T_{-\Delta ,{\rm loc}}(f)(x)\leq T_{-\Delta }(\mathcal{X}_{B(x_j,C_1\rho (x_j))}f)+C\mathcal M_{\rm HL}(f)(x),\quad x\in Q_j.
$$
Let $1<p<\infty$. We can write
\begin{align*}
\int_{\mathbb R^d}|T_{-\Delta ,{\rm loc}}(f) (x)|^pdx&=\sum_{j\in \mathbb N}\int_{Q_j}|T_{-\Delta ,{\rm loc}}(f)(x)|^pdx\\
&\leq C\left(\sum_{j\in \mathbb N} \int_{Q_j}| T_{-\Delta }(\mathcal{X}_{B(x_j,C_1\rho (x_j))}f)(x)|^pdx+\int_{\mathbb R ^d}|\mathcal{M}_{\rm HL}(f)(x)|^pdx\right)\\
&\leq C\left(\sum_{j\in \mathbb N}\int_{B(x_j,C_1\rho (x_j))}|f(x)|^pdx+\int_{\mathbb R^d}|f(x)|^pdx\right)\leq C\int_{\mathbb R^d}|f(x)|^pdx.
\end{align*}

Thus we have proved that $T_{-\Delta,{\rm loc}}$ is bounded from $L^p (\mathbb R^d)$ into itself.

By using \eqref{E1} and Proposition \ref{Prop2.1} (a), proceeding as in \cite[p. 506]{BFHR1} we can deduce that
$$
T_{\mathcal L,{\rm glob}}(f)\leq C\mathcal M_{\rm HL}(f).
$$
Then, $T_{\mathcal L,{\rm glob}}$ is bounded from $L^p(\mathbb R^d)$ into itself.

The arguments in \cite[pp. 507-509]{BFHR1} by using again \eqref{E1} and now Proposition \ref{Prop2.1} (d), allow us to prove that
$$
|T_{\mathcal L,{\rm loc}}(f)-T_{-\Delta ,{\rm loc}}(f)|\leq C\mathcal{M}_{\rm HL}(f).
$$
We conclude that $T_{\mathcal L,{\rm loc}}-T_{-\Delta ,{\rm loc}}$ is bounded from $L^p(\mathbb R^d)$ into itself. By putting together all the above estimates we deduce that the oscillation operator $O(\{t^k\partial _t^kW_t^\mathcal{L}\}_{t>0},\{t_j\}_{j\in \mathbb{Z}})$  is bounded from $L^p(\mathbb R^d)$ into itself.

After proving that $O(\{t^k\partial _t^kW_t^\mathcal{L}\}_{t>0},\{t_j\}_{j\in \mathbb{Z}})$ is bounded from $L^p(\mathbb R^d)$ into itself for every $1<p<\infty $, by using the properties established in Proposition \ref{Prop2.1}, we can proceed as in \cite[pp. 605-609]{TZ} to establish that $O(\{t^k\partial _t^kW_t^\mathcal{L}\}_{t>0},\{t_j\}_{j\in \mathbb{Z}})$ is bounded from $L^p(\mathbb R^d,w)$ into itself, for every $1<p<\infty $ and $w\in A_p^{\rho ,\infty}(\mathbb R^d)$.

We are going to see that the oscillation operator $O(\{t^k\partial _t^kW_t^\mathcal{L}\}_{t>0},\{t_j\}_{j\in \mathbb{Z}})$ is bounded from ${\rm BMO}_{\mathcal L,w}^\alpha (\mathbb R^d)$ into ${\rm BLO}_{\mathcal L,w}^\alpha (\mathbb R^d)$.

By taking into account the weighted $L^p$-boundedness properties of $O(\{t^k\partial _t^kW_t^\mathcal{L}\}_{t>0},\{t_j\}_{j\in \mathbb{Z}})$ that we have just proved and Proposition \ref{Prop2.3}, we can prove by proceeding as in \cite[pp. 610-611 and Lemma 2.1]{TZ} that there exists $C>0$ for which, for every $f\in {\rm BMO}_{\mathcal L,w}^\alpha (\mathbb R^d)$,
\begin{align*}
\int_{B(x_0,r_0)}|O(\{t^k\partial _t^kW_t^\mathcal{L}\}_{t>0},\{t_j\}_{j\in \mathbb{Z}})(f)(x)|dx\\
&\hspace{-4cm}\leq C|B(x_0,r_0)|^\alpha w(B(x_0,r_0))\|f\|_{{\rm BMO}_{\mathcal L,w}^\alpha (\mathbb R^d)},\quad x_0\in \mathbb R^d\mbox{ and }r_0\geq \rho(x_0).
\end{align*}
Note that the last inequality implies that, for every $f\in {\rm BMO}_{\mathcal L,w}^\alpha (\mathbb R^d)$, we have that 
$$
O(\{t^k\partial _t^kW_t^\mathcal{L}\}_{t>0},\{t_j\}_{j\in \mathbb{Z}})(f)(x)<\infty,\quad \mbox{for almost all }x\in \mathbb R ^d.
$$

To finish the proof we need to see that there exists $C>0$ such that, for every $f\in {\rm BMO}_{\mathcal L,w}^\alpha (\mathbb R^d)$,
\begin{align*}
\int_{B(x_0,r_0)}\left(O(\{t^k\partial _t^kW_t^\mathcal{L}\}_{t>0},\{t_j\}_{j\in \mathbb{Z}})(f)(x)-{\rm ess}\hspace{-2mm}\inf_{\hspace{-5mm}y\in B(x_0,r_0)}O(\{t^k\partial _t^kW_t^\mathcal{L}\}_{t>0},\{t_j\}_{j\in \mathbb{Z}})(f)(y)\right)dx&\\
&\hspace{-12cm}\leq C\|B(x_0,r_0)|^\alpha w(B(x_0,r_0))|f\|_{{\rm BMO}_{\mathcal L,w}^\alpha (\mathbb R^d)},\quad x_0\in \mathbb R^d\mbox{ and } 0<r_0<\rho(x_0).
\end{align*}

Let $f\in {\rm BMO}_{\mathcal L,w}^\alpha (\mathbb R^d)$, $x_0\in \mathbb R^d$ and $0<r_0<\rho (x_0)$. We choose $i_0\in \mathbb Z$ such that $t_{i_0}<8r_0^2\leq t_{i_0+1}$.

We define the following sets
\begin{align*}
D_1=\big\{y\in B(x_0,r_0):& \sup_{t_{i_0}\leq \varepsilon _{i_0}<\varepsilon_{i_0+1}\leq t_{i_0+1}}|t^k\partial _t^kW_t^\mathcal L(f)(y)_{|t=\varepsilon _{i_0}}-t^k\partial _t^kW_t^\mathcal L(f)(y)_{|t=\varepsilon _{i_0+1}}|\\
&=\sup_{t_{i_0}\leq \varepsilon _{i_0}<\varepsilon_{i_0+1}\leq 8r_0^2}|t^k\partial _t^kW_t^\mathcal L(f)(y)_{|t=\varepsilon _{i_0}}-t^k\partial _t^kW_t^\mathcal L(f)(y)_{|t=\varepsilon _{i_0+1}}|\big\},
\end{align*}
\begin{align*}
D_2=\big\{y\in B(x_0,r_0):& \sup_{t_{i_0}\leq \varepsilon _{i_0}<\varepsilon_{i_0+1}\leq t_{i_0+1}}|t^k\partial _t^kW_t^\mathcal L(f)(y)_{|t=\varepsilon _{i_0}}-t^k\partial _t^kW_t^\mathcal L(f)(y)_{|t=\varepsilon _{i_0+1}}|\\
&=\sup_{8r_0^2\leq \varepsilon _{i_0}<\varepsilon_{i_0+1}\leq t_{i_0+1}}|t^k\partial _t^kW_t^\mathcal L(f)(y)_{|t=\varepsilon _{i_0}}-t^k\partial _t^kW_t^\mathcal L(f)(y)_{|t=\varepsilon _{i_0+1}}|\big\}
\end{align*}
and
\begin{align*}
D_3=\big\{y\in B(x_0,r_0):& \sup_{t_{i_0}\leq \varepsilon _{i_0}<\varepsilon_{i_0+1}\leq t_{i_0+1}}|t^k\partial _t^kW_t^\mathcal L(f)(y)_{|t=\varepsilon _{i_0}}-t^k\partial _t^kW_t^\mathcal L(f)(y)_{|t=\varepsilon _{i_0+1}}|\\
&=\sup_{t_{i_0}\leq \varepsilon _{i_0}<8r_0^2<\varepsilon_{i_0+1}\leq 8r_0^2}|t^k\partial _t^kW_t^\mathcal L(f)(y)_{|t=\varepsilon _{i_0}}-t^k\partial _t^kW_t^\mathcal L(f)(y)_{|t=\varepsilon _{i_0+1}}|\big\}
\end{align*}

We consider the following decomposition
\begin{align*}
\int_{B(x_0,r_0)}\left(O(\{t^k\partial _t^kW_t^\mathcal{L}\}_{t>0},\{t_j\}_{j\in \mathbb{Z}})(f)(x)-{\rm ess}\hspace{-2mm}\inf_{\hspace{-5mm}y\in B(x_0,r_0)}O(\{t^k\partial _t^kW_t^\mathcal{L}\}_{t>0},\{t_j\}_{j\in \mathbb{Z}})(f)(y)\right)dx&\\
&\hspace{-12cm}=\sum_{i=1}^3H_i, 
\end{align*}
where for every $i=1,2,3$,
$$
H_i=\int_{D_i}\left(O(\{t^k\partial _t^kW_t^\mathcal{L}\}_{t>0},\{t_j\}_{j\in \mathbb{Z}})(f)(x)-{\rm ess}\hspace{-2mm}\inf_{\hspace{-5mm}y\in B(x_0,r_0)}O(\{t^k\partial _t^kW_t^\mathcal{L}\}_{t>0},\{t_j\}_{j\in \mathbb{Z}})(f)(y)\right)dx.
$$
We have that
\begin{align*}
O(\{t^k\partial _t^kW_t^\mathcal{L}\}_{t>0},\{t_j\}_{j\in \mathbb{Z}})(f)(y)\\
&\hspace{-4cm}\geq \left(\sum_{i=i_0+1}^{+\infty} \sup_{t_i\leq \varepsilon _i<\varepsilon _{i+1}\leq t_{i+1}}|t^k\partial _t^kW_t^\mathcal L (f)(y)_{|t=\varepsilon _i}-t^k\partial _t^kW_t^\mathcal L (f)(y)_{|t=\varepsilon _{i+1}}|^2\right)^{1/2},\quad y\in B(x_0,r_0).
\end{align*}
Then,
\begin{align*}
H_1&\leq \int_{D_1} \Big(\sum_{i=-\infty }^{i_0}\sup_{\substack{t_{i}\leq \varepsilon _{i}<\varepsilon_{i+1}\leq t_{i+1}\\\varepsilon _{i+1}\leq 8r_0^2}}|t^k\partial _t^kW_t^\mathcal L(f)(x)_{|t=\varepsilon _{i}}-t^k\partial _t^kW_t^\mathcal L(f)(x)_{|t=\varepsilon _{i+1}}|^2\Big)^{1/2}dx\\
&\quad +\int_{D_1}\Big[\Big(\sum_{i=i_0+1 }^{+\infty }\sup_{t_{i}\leq \varepsilon _{i}<\varepsilon_{i+1}\leq t_{i+1}}|t^k\partial _t^kW_t^\mathcal L(f)(y)_{|t=\varepsilon _{i}}-t^k k\partial _t^kW_t^\mathcal L (f)(y)_{|t=\varepsilon _{i+1}}|^2\Big)^{1/2}\\
&\quad -\Big(\;\;\;{\rm ess}\hspace{-2mm}\inf_{\hspace{-5mm}y\in B(x_0,r_0)}\Big(\sum_{i=i_0+1 }^{+\infty }\sup_{t_{i}\leq \varepsilon _{i}<\varepsilon_{i+1}\leq t_{i+1}}|t^k\partial _t^kW_t^\mathcal L(f)(y)_{|t=\varepsilon _{i}}-t^k k\partial _t^kW_t^\mathcal L (f)(y)_{|t=\varepsilon _{i+1}}|^2\Big)^{1/2}\Big]dx.\\
&\leq \int_{B(x_0,r_0)}\Big(\sum_{i=-\infty }^{i_0}\sup_{\substack{t_{i}\leq \varepsilon _{i}<\varepsilon_{i+1}\leq t_{i+1}\\\varepsilon _{i+1}\leq 8r_0^2}}|t^k\partial _t^kW_t^\mathcal L(f)(x)_{|t=\varepsilon _{i}}-t^k\partial _t^kW_t^\mathcal L (f)(x)_{|t=\varepsilon _{i+1}}|^2\Big)^{1/2}dx\\
&\quad +|B(x_0,r_0)|\;\;{\rm ess}\hspace{-2mm}\sup_{\hspace{-5mm}z,y\in B(x_0,r_0)}\Big(\sum_{i=i_0+1 }^{+\infty }\sup_{t_{i}\leq \varepsilon _{i}<\varepsilon_{i+1}\leq t_{i+1}}|t^k\partial _t^kW_t^\mathcal L(f)(z)_{|t=\varepsilon _{i}}-t^k \partial _t^kW_t^\mathcal L (f)(z)_{|t=\varepsilon _{i+1}}\Big)\\
&\quad -\Big(t^k\partial _t^kW_t^\mathcal L(f)(y)_{|t=\varepsilon_i}-t^k\partial _t^kW_t^\mathcal L (f)(y)_{|t=\varepsilon _{i+1}})|^2\Big)^{1/2}.
\end{align*}
On the other hand, we can write

\begin{align*}
O(\{t^k\partial _t^kW_t^\mathcal{L}\}_{t>0},\{t_j\}_{j\in \mathbb{Z}})(f)(y)\\
&\hspace{-4cm}\geq \left(\sum_{i=i_0}^{+\infty} \sup_{\substack{t_i\leq \varepsilon _i<\varepsilon _{i+1}\leq t_{i+1}\\\varepsilon_i\geq 8r_0^2}}|(t^k\partial _t^kW_t^\mathcal L (f)(y)_{|t=\varepsilon _i}-t^k\partial _t^kW_t^\mathcal L (f)(y)_{|t=\varepsilon _{i+1}}|^2\right)^{1/2},\quad y\in B(x_0,r_0).
\end{align*}

It follows that
\begin{align*}
H_2&\leq \int_{B(x_0,r_0)}\Big(\sum_{i=-\infty }^{i_0-1}\sup_{t_{i}\leq \varepsilon _{i}<\varepsilon_{i+1}\leq t_{i+1}}|t^k\partial _t^kW_t^\mathcal L(f)(y)_{|t=\varepsilon _{i}}-t^k\partial _t^kW_t^\mathcal L (f)(y)_{|t=\varepsilon _{i+1}}|^2\Big)^{1/2}dy\\
&\quad +|B(x_0,r_0)|\;\;{\rm ess}\hspace{-2mm}\sup_{\hspace{-5mm}z,y\in B(x_0,r_0)}\Big(\sum_{i=i_0 }^{+\infty }\sup_{\substack{t_{i}\leq \varepsilon _{i}<\varepsilon_{i+1}\leq t_{i+1}\\\varepsilon_i\geq 8r_0^2}}|(t^k\partial _t^kW_t^\mathcal L(f)(z)_{|t=\varepsilon _{i}}-t^k \partial _t^kW_t^\mathcal L (f)(z)_{|t=\varepsilon _{i+1}})\\
&\quad -(t^k\partial _t^kW_t^\mathcal L(f)(y)_{|t=\varepsilon_i}-t^k\partial _t^kW_t^\mathcal L (f)(y)_{|t=\varepsilon _{i+1}})|^2\Big)^{1/2}.
\end{align*}
Finally, in order to estimate $H_3$, we observe that
\begin{align*}
O(\{t^k\partial _t^kW_t^\mathcal{L}\}_{t>0},\{t_j\}_{j\in \mathbb{Z}})(f)(y)&\\
&\hspace{-4cm}\leq \left(\sum_{i=-\infty }^{i_0-1}\sup_{t_{i}\leq \varepsilon _{i}<\varepsilon_{i+1}\leq t_{i+1}}|t^k\partial _t^kW_t^\mathcal L(f)(y)_{|t=\varepsilon _{i}}-t^k\partial _t^kW_t^\mathcal L (f)(y)_{|t=\varepsilon _{i+1}}|^2\right. \\
&\hspace{-4cm}\quad +\Big(\sup_{t_{i_0}\leq \varepsilon _{i_0}\leq 8r_0^2}|t^k\partial _t^kW_t^\mathcal L(f)(y)_{|t=t_{i_0}}-t^k\partial _t^kW_t^\mathcal L (f)(y)_{|t=8r_0^2}|\\
&\hspace{-4cm}\quad +\sup_{8r_0^2\leq \varepsilon _{i_0+1}\leq t_{i_0+1}}|t^k\partial _t^kW_t^\mathcal L(f)(y)_{|t=8r_0^2}-t^k\partial _t^kW_t^\mathcal L (f)(y)_{|t=t_{i_0+1}}|\Big)^2\\
&\hspace{-4cm}\quad \left.+\sum_{i=i_0+1}^{+\infty }\sup_{t_{i}\leq \varepsilon _{i}<\varepsilon_{i+1}\leq t_{i+1}}|t^k\partial _t^kW_t^\mathcal L(f)(y)_{|t=\varepsilon _{i}}-t^k\partial _t^kW_t^\mathcal L (f)(y)_{|t=\varepsilon _{i+1}}|^2\right)^{1/2}\\
&\hspace{-4cm}\leq \left(\sum_{i=-\infty }^{i_0-1}\sup_{t_{i}\leq \varepsilon _{i}<\varepsilon_{i+1}\leq t_{i+1}}|t^k\partial _t^kW_t^\mathcal L(f)(y)_{|t=\varepsilon _{i}}-t^k\partial _t^kW_t^\mathcal L (f)(y)_{|t=\varepsilon _{i+1}}|^2\right.\\
&\hspace{-4cm}\quad \left.+\sup_{t_{i_0}\leq \varepsilon _{i_0}\leq 8r_0^2}|t^k\partial _t^kW_t^\mathcal L(f)(y)_{|t=\varepsilon_i}-t^k\partial _t^kW_t^\mathcal L (f)(y)_{|t=8r_0^2}|^2\right)^{1/2}\\
&\hspace{-4cm}\quad +\left(\sup_{8r_0^2\leq \varepsilon _{i_0+1}\leq t_{i_0+1}}|t^k\partial _t^kW_t^\mathcal L(f)(y)_{|t=8r_0^2}-t^k\partial _t^kW_t^\mathcal L (f)(y)_{|t=\varepsilon_{i_0+1}}|^2\right.\\
&\hspace{-4cm}\quad \left.+\sum_{i=i_0+1}^{+\infty }\sup_{t_{i}\leq \varepsilon _{i}<\varepsilon_{i+1}\leq t_{i+1}}|t^k\partial _t^kW_t^\mathcal L(f)(y)_{|t=\varepsilon _{i}}-t^k\partial _t^kW_t^\mathcal L (f)(y)_{|t=\varepsilon _{i+1}}|^2\right)^{1/2}\\
&\hspace{-4cm}\leq \left(\sum_{i=-\infty }^{i_0-1}\sup_{t_{i}\leq \varepsilon _{i}<\varepsilon_{i+1}\leq t_{i+1}}|t^k\partial _t^kW_t^\mathcal L(f)(y)_{|t=\varepsilon _{i}}-t^k\partial _t^kW_t^\mathcal L (f)(y)_{|t=\varepsilon _{i+1}}|^2\right.\\
&\hspace{-4cm}\quad \left.+\sup_{t_{i_0}\leq \varepsilon _{i_0}<\varepsilon_{i_0+1}\leq 8r_0^2}|t^k\partial _t^kW_t^\mathcal L(f)(y)_{|t=\varepsilon_{i_0}}-t^k\partial _t^kW_t^\mathcal L (f)(y)_{|t=\varepsilon_{i_0+1}}|^2\right)^{1/2}\\
&\hspace{-4cm}\quad +\left(\sup_{8r_0^2\leq \varepsilon _{i_0+1}\leq t_{i_0+1}}|t^k\partial _t^kW_t^\mathcal L(f)(y)_{|t=8r_0^2}-t^k\partial _t^kW_t^\mathcal L (f)(y)_{|t=\varepsilon_{i_0+1}}|^2\right.\\
&\hspace{-4cm}\quad \left.+\sum_{i=i_0+1}^{+\infty }\sup_{t_{i}\leq \varepsilon _{i}<\varepsilon_{i+1}\leq t_{i+1}}|t^k\partial _t^kW_t^\mathcal L(f)(y)_{|t=\varepsilon _{i}}-t^k\partial _t^kW_t^\mathcal L (f)(y)_{|t=\varepsilon _{i+1}}|^2\right)^{1/2},\quad y\in D_3.
\end{align*}
Thus, we deduce that
\begin{align*}
O(\{t^k\partial _t^kW_t^\mathcal{L}\}_{t>0},\{t_j\}_{j\in \mathbb{Z}})(f)(y)&\\
&\hspace{-4cm}\leq \Big(\sum_{i=-\infty }^{i_0}\sup_{\substack{t_{i}\leq \varepsilon _{i}<\varepsilon_{i+1}\leq t_{i+1}\\\varepsilon_{i+1}\leq 8r_0^2}}|t^k\partial _t^kW_t^\mathcal L(f)(y)_{|t=\varepsilon _{i}}-t^k\partial _t^kW_t^\mathcal L (f)(y)_{|t=\varepsilon _{i+1}}|^2\Big)^{1/2}\\
&\hspace{-4cm}\quad +\left(\sup_{8r_0^2\leq \varepsilon _{i_0+1}\leq t_{i_0+1}}|t^k\partial _t^kW_t^\mathcal L(f)(y)_{|t=8r_0^2}-t^k\partial _t^kW_t^\mathcal L (f)(y)_{|t=\varepsilon_{i+1}}|^2\right.\\
&\hspace{-4cm}\quad \left.+\sum_{i=i_0+1}^{+\infty }\sup_{t_{i}\leq \varepsilon _{i}<\varepsilon_{i+1}\leq t_{i+1}}|t^k\partial _t^kW_t^\mathcal L(f)(y)_{|t=\varepsilon _{i}}-t^k\partial _t^kW_t^\mathcal L (f)(y)_{|t=\varepsilon _{i+1}}|^2\right)^{1/2},\quad y\in D_3.
\end{align*}
On the other hand have that
\begin{align*}
O(\{t^k\partial _t^kW_t^\mathcal{L}\}_{t>0},\{t_j\}_{j\in \mathbb{Z}})(f)(y)\\
&\hspace{-4cm}\geq \left(\sum_{i=i_0+1}^\infty \sup_{t_i\leq \varepsilon _i<\varepsilon _{i+1}\leq t_{i+1}}|(t^k\partial _t^kW_t^\mathcal L (f)(y)_{|t=\varepsilon _i}-t^k\partial _t^kW_t^\mathcal L (f)(y)_{|t=\varepsilon _{i+1}}|^2\right.\\
&\hspace{-4cm}\left.\quad +\sup_{t_{i_0}\leq \varepsilon _{i_0}\leq 8r_0^2<\varepsilon_{i_0+1}\leq t_{i_0+1}}|t^k\partial _t^kW_t^\mathcal L(f)(y)_{|t=\varepsilon_{i_0}}-t^k\partial _t^kW_t^\mathcal L (f)(y)_{|t=\varepsilon_{i_0+1}}|^2\right)^{1/2}\\
&\hspace{-4cm}\geq \left(\sum_{i=i_0+1}^\infty \sup_{t_i\leq \varepsilon _i<\varepsilon _{i+1}\leq t_{i+1}}|(t^k\partial _t^kW_t^\mathcal L (f)(y)_{|t=\varepsilon _i}-t^k\partial _t^kW_t^\mathcal L (f)(y)_{|t=\varepsilon _{i+1}}|^2\right.\\
&\hspace{-4cm}\left.\quad +\sup_{8r_0^2<\varepsilon_{i_0+1}\leq t_{i_0+1}}|t^k\partial _t^kW_t^\mathcal L(f)(y)_{|t=8r_0^2}-t^k\partial _t^kW_t^\mathcal L (f)(y)_{|t=\varepsilon_{i_0+1}}|^2\right)^{1/2},\quad y\in B(x_0,r_0).
\end{align*}
It follows that
\begin{align*}
H_3&\leq \int_{D_3}\Big(\sum_{i=-\infty }^{i_0}\sup_{\substack{t_{i}\leq \varepsilon _{i}<\varepsilon_{i+1}\leq t_{i+1}\\\varepsilon _{i+1}\leq 8r_0^2}}|t^k\partial _t^kW_t^\mathcal L(f)(x)_{|t=\varepsilon _{i}}-t^k\partial _t^kW_t^\mathcal L (f)(x)_{|t=\varepsilon _{i+1}}|^2\Big)^{1/2}dx\\
&\quad +|B(x_0,r_0)|\;\;{\rm ess}\hspace{-2mm}\sup_{\hspace{-5mm}z,y\in B(x_0,r_0)}\Big(\sum_{i=i_0 +1}^{+\infty }\sup_{t_{i}\leq \varepsilon _{i}<\varepsilon_{i+1}\leq t_{i+1}}|(t^k\partial _t^kW_t^\mathcal L(f)(z)_{|t=\varepsilon _{i}}-t^k \partial _t^kW_t^\mathcal L (f)(z)_{|t=\varepsilon _{i+1}})\\
&\quad -(t^k\partial _t^kW_t^\mathcal L(f)(y)_{|t=\varepsilon_i}-t^k\partial _t^kW_t^\mathcal L (f)(y)_{|t=\varepsilon _{i+1}})|^2\\
&\quad +\sup_{8r_0^2<\varepsilon_{i_0+1}\leq t_{i_0+1}}|(t^k\partial _t^kW_t^\mathcal L(f)(z)_{|t=8r_0^2}-t^k\partial _t^kW_t^\mathcal L (f)(z)_{|t=\varepsilon_{i_0+1}})\\
&\quad -(t^k\partial _t^kW_t^\mathcal L(f)(y)_{|t=8r_0^2}-t^k\partial _t^kW_t^\mathcal L (f)(y)_{|t=\varepsilon_{i_0+1}})|^2\Big)^{1/2}\\
&\leq \int_{D_3}\Big(\sum_{i=-\infty }^{i_0}\sup_{\substack{t_{i}\leq \varepsilon _{i}<\varepsilon_{i+1}\leq t_{i+1}\\\varepsilon _{i+1}\leq 8r_0^2}}|t^k\partial _t^kW_t^\mathcal L(f)(x)_{|t=\varepsilon _{i}}-t^k\partial _t^kW_t^\mathcal L (f)(x)_{|t=\varepsilon _{i+1}}|^2\Big)^{1/2}dx\\
&\quad +|B(x_0,r_0)|\;\;{\rm ess}\hspace{-2mm}\sup_{\hspace{-5mm}z,y\in B(x_0,r_0)}\Big(\sum_{i=i_0 +1}^{+\infty }\sup_{t_{i}\leq \varepsilon _{i}<\varepsilon_{i+1}\leq t_{i+1}}|(t^k\partial _t^kW_t^\mathcal L(f)(z)_{|t=\varepsilon _{i}}-t^k \partial _t^kW_t^\mathcal L (f)(z)_{|t=\varepsilon _{i+1}})\\
&\quad -(t^k\partial _t^kW_t^\mathcal L(f)(y)_{|t=\varepsilon_i}-t^k\partial _t^kW_t^\mathcal L (f)(y)_{|t=\varepsilon _{i+1}})|^2\\
&\quad +\sup_{8r_0^2\leq \varepsilon _{i_0}<\varepsilon_{i_0+1}\leq t_{i_0+1}}|(t^k\partial _t^kW_t^\mathcal L(f)(z)_{|t=\varepsilon_{i_0}}-t^k\partial _t^kW_t^\mathcal L (f)(z)_{|t=\varepsilon_{i_0+1}})\\
&\quad -(t^k\partial _t^kW_t^\mathcal L(f)(y)_{|t=8r_0^2}-t^k\partial _t^kW_t^\mathcal L (f)(y)_{|t=\varepsilon_{i_0+1}})|^2\Big)^{1/2}.
\end{align*}
Thus,
\begin{align*}
H_3&\leq \int_{B(x_0,r_0)}\Big(\sum_{i=-\infty }^{i_0}\sup_{\substack{t_{i}\leq \varepsilon _{i}<\varepsilon_{i+1}\leq t_{i+1}\\\varepsilon _{i+1}\leq 8r_0^2}}|t^k\partial _t^kW_t^\mathcal L(f)(x)_{|t=\varepsilon _{i}}-t^k\partial _t^kW_t^\mathcal L (f)(x)_{|t=\varepsilon _{i+1}}|^2\Big)^{1/2}dx\\
&\quad +|B(x_0,r_0)|\;\;{\rm ess}\hspace{-2mm}\sup_{\hspace{-5mm},z,y\in B(x_0,r_0)}\Big(\sum_{i=i_0}^{+\infty }\sup_{\substack{t_{i}\leq \varepsilon _{i}<\varepsilon_{i+1}\leq t_{i+1}\\\varepsilon _{i}\geq 8r_0^2}}|(t^k\partial _t^kW_t^\mathcal L(f)(z)_{|t=\varepsilon _{i}}-t^k \partial _t^kW_t^\mathcal L (f)(z)_{|t=\varepsilon _{i+1}})\\
&\quad -(t^k\partial _t^kW_t^\mathcal L(f)(y)_{|t=\varepsilon_i}-t^k\partial _t^kW_t^\mathcal L (f)(y)_{|t=\varepsilon _{i+1}})|^2\Big)^{1/2}.
\end{align*}
In order to get our objective it is sufficient to prove that
\begin{align}\label{E3}
\int_{B(x_0,r_0)}\Big(\sum_{i=-\infty }^{i_0}\sup_{\substack{t_{i}\leq \varepsilon _{i}<\varepsilon_{i+1}\leq t_{i+1}\\\varepsilon _{i+1}\leq 8r_0^2}}|t^k\partial _t^kW_t^\mathcal L(f)(x)_{|t=\varepsilon _{i}}-t^k\partial _t^kW_t^\mathcal L (f)(x)_{|t=\varepsilon _{i+1}}|^2\Big)^{1/2}dx\nonumber\\
&\hspace{-6cm}\leq C|B(x_0,r_0)|^\alpha w(B(x_0,r_0))\|f\|_{{\rm BMO}_{\mathcal L,w}^\alpha (\mathbb R^d)}, 
\end{align}
and 
\begin{align}\label{E4}
{\rm ess}\hspace{-2mm}\sup_{\hspace{-5mm}x,y\in B(x_0,r_0)}\Big(\sum_{i=i_0}^{+\infty }\sup_{\substack{t_{i}\leq \varepsilon _{i}<\varepsilon_{i+1}\leq t_{i+1}\\\varepsilon _{i}\geq 8r_0^2}}|(t^k\partial _t^kW_t^\mathcal L(f)(x)_{|t=\varepsilon _{i}}-t^k \partial _t^kW_t^\mathcal L (f)(x)_{|t=\varepsilon _{i+1}})\nonumber\\
&\hspace{-10cm}\quad -(t^k\partial _t^kW_t^\mathcal L(f)(y)_{|t=\varepsilon_i}-t^k\partial _t^kW_t^\mathcal L (f)(y)_{|t=\varepsilon _{i+1}})|^2\Big)^{1/2}\nonumber\\
&\hspace{-10cm}\leq C|B(x_0,r_0)|^{\alpha -1}w(B(x_0,r_0)\|f\|_{{\rm BMO}_{\mathcal L,w}^\alpha (\mathbb R^d)}.
\end{align}

By using \eqref{E1} we can get \eqref{E3} and \eqref{E4} by proceeding as in the proof of \eqref{*1} and \eqref{*2}, respectively.

\section{Proof of Theorem \ref{Th1.2} for the operator $SV(\{t^k\partial_t^kW_t^\mathcal{L}\}_{t>0})$}\label{S5}

In order to prove Theorem \ref{Th1.2} for the short variation operator $SV(\{t^k\partial_t^kW_t^\mathcal{L}\}_{t>0})$ we can proceed as in the previous section for the oscillation operator $O(\{t^k\partial_t^kW_t^\mathcal{L}\}_{t>0},\{t_j\}_{j\in \mathbb{Z}})$.

Note firstly that if $F$ is a derivable function in $(0,\infty )$ we have that
\begin{align}\label{3.1}
SV(\{F(t)\}_{t>0})\leq C\int_0^\infty |F'(t)|dt.
\end{align}
By taking into account \eqref{E2} and according to \cite[Lemma 2.4, (3)]{CJRW1} it follows that the operator $SV(\{t^k\partial_t^kW_t\}_{t>0})$ is bounded from $L^p(\mathbb R^d)$ into itself, for every $1<p<\infty$.

We now define the local and global operators as in Section \ref{S4}. We have that
\begin{align*}
SV(\{t^k\partial_t^kW_t^\mathcal{L}\}_{t>0})(f)&\leq SV_{{\rm loc}}(\{t^k\partial_t^k(W_t^\mathcal{L}-W_t)\}_{t>0})(f)+SV_{{\rm loc}}(\{t^k\partial_t^kW_t\}_{t>0})(f)\\
&\quad +SV_{{\rm glob}}(\{t^k\partial_t^kW_t^\mathcal{L}\}_{t>0})(f).
\end{align*}
Then, by proceeding as in the study of the oscillation operator in the previous section we can see that $SV(\{t^k\partial_t^kW_t^\mathcal{L}\}_{t>0})$ is bounded from $L^p(\mathbb R^d)$ into itself, for every $1<p<\infty$. By using \eqref{3.1} the arguments in \cite[pp. 605-609]{TZ} allow us to see that the operator $SV\{t^k\partial_t^kW_t^\mathcal{L}\}_{t>0})$ is bounded from $L^p(\mathbb R^d,w)$ into itself, for every $1<p<\infty$ and $w\in A_p^{\rho , \infty}(\mathbb{R}^d)$.

Let now $x_0\in \mathbb R^d$ and $r_0>0$ such that $r_0<\rho (x_0)$. We choose $k_0\in \mathbb{N}$ such that $2^{-k_0}<8r_0^2\leq 2^{-k_0+1}$. We have that
\begin{align*}
V_{k_0}(\{t^k\partial_t^kW_t^\mathcal{L}\}_{t>0})(f)(x)&\\
&\hspace{-2cm}=\sup_{\substack{2^{-k_0}<t_n<...<t_1\leq 2^{-k_0+1}\\n\in \mathbb{N}}}\Big(\sum_{j=1}^{n-1}|t^k\partial_t^kW_t^\mathcal{L}(f)(x)_{|t=t_j}-t^k\partial_t^kW_t^\mathcal{L}(f)(x)_{|t=t_{j+1}}|^2\Big)^{1/2}\\
&\hspace{-2cm}\leq \sup_{\substack{2^{-k_0}<s_\ell<...<s_1\leq 8r_0^2\\ \ell \in \mathbb{N}}}\Big(\sum_{j=1}^{\ell-1}|t^k\partial_t^kW_t^\mathcal{L}(f)(x)_{|t=s_j}-t^k\partial_t^kW_t^\mathcal{L}(f)(x)_{|t=s_{j+1}}|^2\Big)^{1/2}\\
&\hspace{-2cm}\quad +\sup_{\substack{8r_0^2\leq s_\ell<...<s_1\leq 2^{-k_0+1}\\ \ell \in \mathbb{N}}}\Big(\sum_{j=1}^{\ell-1}|t^k\partial_t^kW_t^\mathcal{L}(f)(x)_{|t=s_j}-t^k\partial_t^kW_t^\mathcal{L}(f)(x)_{|t=s_{j+1}}|^2\Big)^{1/2}\\
&\hspace{-2cm}=: V_{k_0,-}(\{t^k\partial_t^kW_t^\mathcal{L}\}_{t>0})(f)(x)+V_{k_0,+}(\{t^k\partial_t^kW_t^\mathcal{L}\}_{t>0})(f)(x),\quad x\in \mathbb{R}^d,
\end{align*}
and 
$$
V_{k_0}(\{t^k\partial_t^kW_t^\mathcal{L}\}_{t>0})(f)(x)\geq V_{k_0,+}(\{t^k\partial_t^kW_t^\mathcal{L}\}_{t>0})(f)(x),\quad x\in \mathbb{R}^d.
$$
It follows that
\begin{align*}
\int_{B(x_0,r_0)}(SV(\{t^k\partial_t^kW_t^\mathcal{L}\}_{t>0})(f)(x)-{\rm ess}\hspace{-2mm}\inf_{\hspace{-5mm}y\in B(x_0,r_0)}SV(\{t^k\partial_t^kW_t^\mathcal{L}\}_{t>0})(f)(y))dx\\
&\hspace{-11.3cm}\leq \int_{B(x_0,r_0)}\left(\sum_{j=k_0+1}^\infty (V_j(\{t^k\partial_t^kW_t^\mathcal{L}\}_{t>0})(f)(x))^2+(V_{k_0,-}(\{t^k\partial_t^kW_t^\mathcal{L}\}_{t>0})(f)(x))^2\right)^{1/2}dx\\
&\hspace{-11.3cm}\quad +|B(x_0,r_0)|\;\;\;{\rm ess}\hspace{-2mm}\sup_{\hspace{-4mm}z,y\in B(x_0,r_0)}
\Big|\Big((V_{k_0,+}(\{t^k\partial_t^kW_t^\mathcal{L}\}_{t>0})(f)(z))^2+\sum_{j=-\infty}^{k_0-1}(V_j(\{t^k\partial_t^kW_t^\mathcal{L}\}_{t>0})(f)(z))^2\Big)^{1/2}\\
&\hspace{-11.3cm}\quad -\Big((V_{k_0,+}(\{t^k\partial_t^kW_t^\mathcal{L}\}_{t>0})(f)(y))^2+\sum_{j=-\infty}^{k_0-1}(V_j(\{t^k\partial_t^kW_t^\mathcal{L}\}_{t>0})(f)(y))^2\Big)^{1/2}\Big|.
\end{align*}
We have all the ingredients to finish the proof by proceeding as in Section \ref{S3}.

\section{Proof of Theorem \ref{Th1.1}}\label{S6}

We firstly establish that the maximal operator $W_*^{\mathcal{L},k}$ is bounded from $L^p(\mathbb R^d,w)$ into itself. In order to do this, it is sufficient to proceed as in the proof of \cite[Theorem 2]{BHS2} by using Proposition \ref{Prop2.1}, (a).

Let $f\in {\rm BMO}_{\mathcal L,w}^\alpha (\mathbb R ^d)$ and $x_0\in \mathbb R^d$. Taking $r_0=\rho (x_0)$, we decompose $f$ as follows $f=f\mathcal X_{B(x_0,2r_0)}+f\mathcal X_{B(x_0,2r_0)^c}=:f_1+f_2$. Since $w\in A_p^{\rho ,\theta}(\mathbb{R}^d)$ we have that $w^{-1/(p-1)}\in A_{p'}^{\rho ,\theta}(\mathbb{R}^d)$ (Proposition \ref{Prop2.3}, (a)). H\"older's inequality and Propositions \ref{Prop2.3}, (c), and \ref{Prop2.4} lead to
\begin{align*}
  \int_{B(x_0,r_0)}|W_*^{\mathcal{L},k}(f_1)(x)|dx&\leq w(B(x_0,r_0))^{1/p}\left(\int_{B(x_0,r_0)}|W_*^{\mathcal{L},k}(f_1)(x)|^{p'}w^{-\frac{1}{p-1}}(x)dx\right)^{1/p'}\\
    &\leq Cw(B(x_0,r_0))^{1/p}\left(\int_{B(x_0,2r_0)}|f(x)|^{p'}w^{-\frac{1}{p-1}}(x)dx\right)^{1/p'}\\
    &\leq C|B(x_0,r_0)|^\alpha w(B(x_0,r_0))\|f\|_{{\rm BMO}_{\mathcal{L},w}^\alpha (\mathbb{R}^d)}.
\end{align*}
On the other hand, by using Proposition \ref{Prop2.1}, for every $N\in \mathbb{N}$ we can find $C=C(N)>0$ such that, for each $x\in B(x_0,r_0)$,
\begin{align*}
|W_*^{\mathcal L,k}(f_2)(x)|&\leq C\sup_{t>0}\Big(\frac{\sqrt{t}}{\rho (x)}\Big)^{-N}\frac{1}{t^{d/2}}\int_{\mathbb R^d\setminus B(x_0,2r_0)}e^{-c\frac{|x-y|^2}{t}}|f(y)|dy\\
&\leq C\rho (x_0)^N\int_{\mathbb R^d\setminus B(x_0,2r_0)}\frac{|f(y)|}{|x_0-y|^{N+d}}|f(y)|dy\\
&\leq C\rho (x_0)^N\sum_{j=1}^\infty \frac{1}{(2^j\rho (x_0))^{N+d}}\int_{B(x_0,2^{j+1}}r_0)|f(y)|dy\\
&\leq \frac{C}{r_0^d}\|f\|_{{\rm BMO}_{\mathcal L,w}^\alpha (\mathbb R^d)}\sum_{j=1}^\infty \frac{1}{2^{j(N+d)}}|B(x_0,2^{j+1}r_0)|^\alpha w(B(x_0,2^{j+1}r_0))\\
&\leq C\|f\|_{{\rm BMO}_{\mathcal L,w}^\alpha (\mathbb R^d)}|B(x_0,r_0)|^{\alpha -1}w(B(x_0,r_0))\sum_{j=1
}^\infty 2^{j(\alpha d+(\theta +d)p-d-N)}.
\end{align*}
In the last inequality we have used Proposition \ref{Prop2.3}, (b). By taking $N\in \mathbb{N}$, $N>d(p+\alpha-1)+p\theta $ we obtain
$$
|W_*^{\mathcal L,k}(f_2)(x)|\leq C|B(x_0,r_0)|^{\alpha -1}w(B(x_0,r_0))\|f\|_{{\rm BMO}_{\mathcal L,w}^\alpha (\mathbb R^d)},\quad x\in B(x_0,r_0).
$$
Then,
$$
\int_{B(x_0,r_0)}|W_*^{\mathcal L,k}(f_2)(x)|dx\leq C|B(x_0,r_0)|^\alpha w(B(x_0,r_0))\|f\|_{{\rm BMO}_{\mathcal L,w}^\alpha (\mathbb R^d)},
$$
and we conclude that
\begin{equation}\label{6.1}
\int_{B(x_0,r_0)}|W_*^{\mathcal L,k}(f)(x)|dx\leq C|B(x_0,r_0)|^\alpha w(B(x_0,r_0))\|f\|_{{\rm BMO}_{\mathcal L,w}^\alpha (\mathbb R^d)}.
\end{equation}
From \eqref{6.1} we deduce that $W_*^{\mathcal L,k}(f)(x)<\infty $, for almost all $x\in \mathbb R ^d$.

Let now $x_0\in \mathbb R^d$ and $0<r_0<\rho (x_0)$. We are going to see that 
$$
\int_{B(x_0,r_0)}(W_*^{\mathcal L,k}(f)(x)-{\rm ess}\hspace{-2mm}\inf_{\hspace{-5mm}y\in B(x_0,r_0)}W_*^{\mathcal L,k}(f)(y))dx\leq C|B(x_0,r_0)|^\alpha w(B(x_0,r_0))\|f\|_{{\rm BMO}_{\mathcal L,w}^\alpha (\mathbb R ^d)}.
$$
In order to do this we adapt the ideas developed in Section \ref{S3}. We have that
\begin{align*}
\int_{B(x_0,r_0)}(W_*^{\mathcal L,k}(f)(x)-{\rm ess}\hspace{-2mm}\inf_{\hspace{-5mm}y\in B(x_0,r_0)}W_*^{\mathcal L,k}(f)(y))dx&\leq \int_{B(x_0,r_0)}\;\;\sup_{0<t<8r_0^2}|t^k\partial _t^kW_t^{\mathcal L}
(f)(x)|dx\\
&\hspace{-5cm}\quad +|B(x_0,r_0)|{\rm ess}\hspace{-2mm}\sup_{\hspace{-5mm}z,y\in B(x_0,r_0)}\sup_{t\geq 8r_0^2}|t^k\partial _t^kW_t^\mathcal{L} (f)(z)-t^k\partial _t^kW_t^\mathcal{L}(f)(y)|\\
&\hspace{-5cm} =: M_1(f)+M_2(f).
\end{align*}
We decompose $f$ as follows 
$$
f=(f-f_{B(x_0,r_0)})\mathcal {X}_{B(x_0,2r_0)}+(f-f_{B(x_0,r_0)})\mathcal{X
}_{B(x_0,2r_0)^c}+f_{B(x_0,r_0)}=:f_1+f_2+f_3.
$$
Since $W_*^{\mathcal{L},k}$ is bounded from $L^{p'}(\mathbb R^d,w^{-1/(p-1)})$ into itself we get
$$
M_1(f_1)\leq C|B(x_0,r_0)|^\alpha w(B(x_0,r_0)\|f\|_{{\rm BMO}_{\mathcal L,w}^\alpha (\mathbb R ^d)}.
$$
According to Proposition \ref{Prop2.1}, (a), we obtain
\begin{align*}
M_1(f_2)&\leq C\int_{B(x_0,r_0)}\int_{\mathbb R^d\setminus B(x_0,2r_0)}|f(y)-f_{B(x_0,r_0)}|\sup_{0<t<8r_0^2}\frac{1}{t^{d/2}}e^{-c\frac{|x-y|^2}{t}}dydx\\
&\leq C\int_{\mathbb R^d\setminus B(x_0,2r_0)}|f(y)-f_{B(x_0,r_0)}|\frac{e^{-c\frac{|x_0-y|^2}{r_0^2}}}{|x_0-y|^d}dydx\\
&\leq C|B(x_0,r_0)|^\alpha w(B(x_0,r_0))\|f\|_{{\rm BMO}_{\mathcal L,w}^\alpha (\mathbb R ^d)}.
\end{align*}
Suppose now $k\in \mathbb{N}$, $k\geq 1$. Since $\partial _t^kW_t(1)=0$, it follows that
\begin{align*}
M_1(f_3)&\leq |f_{B(x_0,r_0)}|\int_{B(x_0,r_0)}\;\;\sup_{0<t<8r_0^2}\Big|\int_{\mathbb R^d}t^k\partial _t^k[W_t^{\mathcal L}(x,y)-W_t(x-y)]dy\Big|dx\\
&\leq |f_{B(x_0,r_0)}|\int_{B(x_0,r_0)}\sup_{0<t<8r_0^2}\int_{|x-y|<\rho (x_0)}|t^k\partial _t^k[W_t^{\mathcal L}(x,y)-W_t(x-y)]|dydx\\
&\quad +|f_{B(x_0,r_0)}|\int_{B(x_0,r_0)}\sup_{0<t<8r_0^2}\int_{|x-y|\geq\rho (x_0)}|t^k\partial _t^k[W_t^{\mathcal L}(x,y)-W_t(x-y)]|dydx\\
&=: M_{11}(f_3)+M_{12}(f_3).
\end{align*}
According to Proposition \ref{Prop2.1}, (d), since $2-\frac{d}{q}>d(p+\alpha-1)+p\theta$, we obtain
\begin{align*}
M_{11}(f_3)&\leq C|f_{B(x_0,r_0)}|\int_{B(x_0,r_0)}\;\;\sup_{0<t<8r_0^2}\int_{|x-y|<\rho (x_0)}\Big(\frac{\sqrt{t}}{\rho (x_0)}\Big)^{2-\frac{d}{q}}\frac{e^{-c\frac{|x-y|^2}{t}}}{t^{d/2}}dydx\\
&\leq C|f_{B(x_0,r_0)}|\int_{B(x_0,r_0)}\int_{|x-y|<\rho (x_0)}\rho (x_0)^{\frac{d}{q}-2}\frac{e^{-c\frac{|x-y|^2}{r_0^2}}}{|x-y|^{d+\frac{d}{q}-2}}dydx\\
&\leq C|B(x_0,r_0)|^\alpha w(B(x_0,r_0))\|f\|_{{\rm BMO}_{\mathcal L,w}^\alpha (\mathbb R ^d)}.
\end{align*}
By using again Proposition \ref{Prop2.1}, (a), and (\ref{dergaus}), for every $\beta >0$, we get
$$
\int_{B(x_0,r_0)}\;\;\sup_{0<t<8r_0^2}\int_{|x-y|\geq \rho (x_0)}|t^k\partial _t^k[W_t^{\mathcal L}(x,y)-W_t(x-y)]|dydx\leq C|B(x_0,r_0)|\Big(\frac{r_0}{\rho (x_0)}\Big)^\beta.
$$
By taking $\beta =d(\alpha +p-1)+p\theta$ it follows that
$$
M_{12}(f_3)\leq C|B(x_0,r_0)|^\alpha w(B(x_0,r_0))\|f\|_{{\rm BMO}_{\mathcal L,w}^\alpha (\mathbb R ^d)}.
$$
We conclude that
$$
M_1(f_3)\leq C|B(x_0,r_0)|^\alpha w(B(x_0,r_0))\|f\|_{{\rm BMO}_{\mathcal L,w}^\alpha (\mathbb R ^d)}.
$$
We can write 
\begin{align*}
M_2(f)&\leq |B(x_0,r_0)|{\rm ess}\hspace{-2mm}\sup_{\hspace{-5mm}x,y\in B(x_0,r_0)}\sup_{t>8\rho (x_0)^2}\Big|\int_{\mathbb R^d}[t^k\partial _t^kW_t^\mathcal{L} (x,z)-t^k\partial _t^kW_t^\mathcal{L}(y,z)]f(z)dz\Big|\\
&\quad +|B(x_0,r_0)|{\rm ess}\hspace{-2mm}\sup_{\hspace{-5mm}x,y\in B(x_0,r_0)}\sup_{8r_0^2\leq t<8\rho (x_0)^2}\Big|\int_{\mathbb R^d}[t^k\partial _t^kW_t^\mathcal{L} (x,z)-t^k\partial _t^kW_t^\mathcal{L}(y,z)]f(z)dz\Big|\\
&=: M_{21}(f)+M_{22}(f).
\end{align*}
By using Proposition \ref{Prop2.1}, (b), for every $0<\delta <\delta_0$ there exists $C>0$ such that
\begin{align*}
\Big|\int_{\mathbb R^d}[t^k\partial _t^kW_t^\mathcal{L} (x,z)-t^k\partial _t^k & W_t^\mathcal{L}(y,z)]f(z)dz\Big|\\ &\leq C\Big(\frac{|x-y|}{\sqrt{t}}\Big)^\delta\|f\|_{{\rm BMO}_{\mathcal L,w}^\alpha(\mathbb R^d)} \frac{w(B(x_0,r_0))}{r_0^{p(\theta+d)}}t^{\frac{d}{2}(p+\alpha -1)+\frac{p\theta}{2}},
\end{align*}
for each $t>8\rho (x_0)^2$ and $x,y\in B(x_0,r_0)$.
Then,
$$
M_{21}(f)\leq C|B(x_0,r_0)|^\alpha w(B(x_0,r_0))\|f\|_{{\rm BMO}_{\mathcal L,w}^\alpha (\mathbb R ^d)},
$$
provided that $\delta >d(p+\alpha -1)+p\theta$.

On the other hand, we have that
\begin{align*}
\Big|\int_{\mathbb R^d}[t^k\partial _t^kW_t^\mathcal{L} (x,z)-t^k\partial _t^kW_t^\mathcal{L}(y,z)]f(z)dz\Big|&\\
&\hspace{-4cm}\leq \Big|\int_{\mathbb R^d}[t^k\partial _t^kW_t^\mathcal{L} (x,z)-t^k\partial _t^kW_t^\mathcal{L}(y,z)](f(z)-f_{B(x_0,r_0)})dz\Big|\\
&\hspace{-4cm}\quad +\Big|\int_{\mathbb R^d}[t^k\partial _t^kW_t^\mathcal{L} (x,z)-t^k\partial _t^kW_t^\mathcal{L}(y,z)]dz\Big||f_{B(x_0,r_0)}|\\
&\hspace{-4cm}=:H_1(x,y,t)+H_2(x,y,t),\quad x,y\in B(x_0,r_0)\mbox{ and }t\in (8r_0^2,8\rho (x_0)^2).
\end{align*}
We get 
$$
\sup_{r_0^2<t\leq 8\rho (x_0)^2}(H_1(x,y,t)+H_2(x,y,t))\leq Cw(B(x_0,r_0))r_0^{d(\alpha -1)}\|f\|_{{\rm BMO}_{\mathcal L,w}^\alpha (\mathbb R^d)
}.
$$
We conclude that
$$
M_{22}(f)\leq C|B(x_0,r_0)|^\alpha w(B(x_0,r_0))\|f\|_{{\rm BMO}_{\mathcal L,w}^\alpha (\mathbb R ^d)}.
$$
Thus,
$$
M_2(f)\leq C|B(x_0,r_0)|^\alpha w(B(x_0,r_0))\|f\|_{{\rm BMO}_{\mathcal L,w}^\alpha (\mathbb R ^d)},
$$
and the proof is finished when $k\in \mathbb{N}$, $k\geq 1$.

In order to establish the result for $k=0$, that is, to see that the maximal operator $W_*^{\mathcal L}$ is bounded from ${\rm BMO}_{\mathcal L,w}^\alpha (\mathbb R ^d)$ into ${\rm BLO}_{\mathcal L,w}^\alpha (\mathbb R ^d)$ we can proceed as in the proof of \cite[Theorem 3.1]{YYZ2}. We remark that the arguments in the proof of \cite[Theorem 3.1]{YYZ2} can be adapted to establish that the maximal operator $W_*^{\mathcal L,k}$, $k\in \mathbb{N}$, $k\geq 1$, is bounded from ${\rm BMO}_{\mathcal L,w}^\alpha (\mathbb R ^d)$ into ${\rm BLO}_{\mathcal L,w}^\alpha (\mathbb R ^d)$ but we have preferred to show that the procedure in Section \ref{S3} also works for $W_*^{\mathcal L,k}$, $k\in \mathbb N$, $k\geq 1$.
\bibliographystyle{acm}

\end{document}